\documentclass[preprint,review,12pt]{elsarticle}
 
\usepackage{amssymb}

\usepackage[]{times}
\usepackage{epsfig}
\usepackage{subfigure}
\usepackage{amsmath,amsthm}
\usepackage{bm,amssymb}
\usepackage{graphicx}
\usepackage{multirow}
\usepackage{flushend}
\usepackage{url}
\usepackage{color}
\newenvironment{myindentpar}[1]%
 {\begin{list}{}%
         {\setlength{\leftmargin}{#1}}%
         \item[]%
 }
 {\end{list}}
\newtheorem{theorem}{Theorem}[section]

\journal{Energy Conversion and Management}

\begin{document}

\begin{frontmatter}


\author{Nikolai~A.~Sinitsyn}
 \ead{nsinitsyn@lanl.gov}
 \address{Los Alamos National Laboratory, Los Alamos, NM 87545, USA}
 \author{Soumya Kundu\corref{cor1}}
 \ead{soumyak@umich.edu}
 \cortext[cor1]{Corresponding author}
\address{University of Michigan, Ann Arbor, MI 48105, USA}
\author{Scott Backhaus}
 \ead{backhaus@lanl.gov}
\address{Los Alamos National Laboratory, Los Alamos, NM 87545, USA}

\title{Safe Protocols for Generating Power Pulses with Heterogeneous Populations of Thermostatically Controlled Loads}


\author{}

\address{}

\begin{abstract}
We explore methods to use thermostatically controlled loads (TCLs), such as water heaters and air conditioners, to provide ancillary services by assisting in balancing generation and load. We show that by adding simple imbedded instructions and a small amount of memory to temperature controllers of TCLs,
it is possible to design open-loop control algorithms capable of creating short-term pulses of demand response without unwanted power oscillations associated with temporary synchronization of the TCL dynamics.  By moving a small amount of intelligence to each of the end point TCL devices, we are able to leverage our knowledge of the time dynamics of TCLs to shape the demand response pulses for different power system applications.  A significant benefit of our open-loop method is the reduction from two-way to one-way broadcast communication which also eliminates many basic consumer privacy issues.  In this work, we focus on developing the algorithms to generate a set of fundamental pulse shapes that can subsequently be used to create demand response with arbitrary profiles.  Demand response control methods, such as the one developed here, open the door to fast, nonperturbative control of large aggregations of TCLs.
\end{abstract}


\end{frontmatter}


\section{Introduction}
\label{sec:intro}
Conventional thermal generators, which are constrained by physical ramp rates, have difficulty in manoeuvering to compensate for variability in loads or other generation, e.g. variability due to time-intermittent generation such as wind 
generation. On the other hand, electrical loads, e.g. thermostatically-controlled-loads (TCLs), may be a better alternative to provide the required generation-balancing ancillary services 
\cite{strbac,klobasa,heffner,tcl1,tcl2,tcl3,tcl4,tcl5,tcl6,tcl7,tcl8,tcl9,tcl10,zehir12,bel09} because it is feasible for these loads to compensate for power imbalances much more quickly than thermal generators.  TCLs such as air conditioners (AC) and water 
heaters account for about 50\% of electricity consumption in the United States \cite{eia_recs01} implying that the potential resource provided by TCLs is quite large. A residential AC unit consumes 0.5-7 kW, and when aggregated over a 
large city with a million houses, the total controllable resource approaches $\sim$3 GW. The typical TCL duty cycle will limit the amount of resource available at any one time, but even a relatively low duty cycle of 10\% still yields a 
controllable resource of $\sim$300MW.  Because these TCL loads are on/off loads and not subject to the ramping limits of conventional thermal generation, they can be fully deployed on the time scale of the latency of the communication 
system used to access them, e.g. $\sim$0.3 minutes \cite{heffner,kirby08,kirby07,kirby,sastry10}. Such short time scales for full deployment are not achievable with spinning or non-spinning reserves and, therefore, TCLs have 
considerable potential to provide both system reserves and generation-load balancing ancillary services \cite{callaway,hiskens}.


Past research on controlling TCLs has been either on direct interruption for long (several hours) duration \cite{kirby07,chong,mortensen,malhame} or on shifting of the TCLs temperature set point \cite{callaway,pscc,perfumo,bashash}. The infrequent application of the former technique is useful for critical peak shaving, but frequent application would quickly lead to customer fatigue.  An approach based on small shifts of the TCL set point appears more attractive. In this case, subtle (0.1-0.5$^o$C) changes in the thermostat set point temperature can create significant changes in aggregate power consumption. Such small set point shifts would be almost unnoticed by customers and could be applied frequently as long there is not a long-term drift in one direction.  However, this method is not without its difficulties. Extensive analytical and numerical modeling studies have shown that suddenly imposed changes to the operation cycles of a large population of TCL loads leads to temporary synchronization of otherwise uncorrelated autonomous devices resulting in potentially detrimental power oscillations \cite{callaway,pscc,perfumo,bashash,koch}. Efforts to resolve this problem by feedback control techniques encounter additional problems, e.g. the need for two-way communication to provide full state information for each TCL  and subsequent individual control of all TCLs. This approach, at least, requires development of new software that would process large amount of data on-fly while the state estimation must be done in a noisy environment created by all the other loads on the system \cite{mathieu12}.  Perhaps more important is that two-way communication  leads to privacy security issues that have already raised a popular protest \cite{protest1,protest2,protest3,protest4}. It seems that it can be implemented only on the voluntary basis, possibly with only a small fraction of TCL owners willing to lower their security level in return for rebate. For example, if only 5\% from 1 million TCL owners in the area allow for the two-way communication with their utility company, this would lead to insignificant available power for control. However, the information collected from this  5\% of the population should be sufficient to estimate some important aggregated parameters of the full population, such as the fraction of the currently working TCLs, the total power consumed by these TCLs, distribution of the temperature band widths of their operation, fraction of TCLs in the ON and OFF modes e.t.c. Extraction of this information, arriving to the utility company from only 5\% of all TCLs, does not require a complex software and can be achieved "on-fly".  Therefore, we believe that the technology that would allow for reliable one-way communication control of the remaining "95\% of TCLs", and that would be based on the knowledge of only basic aggregate parameters of the full population of TCLs, should be attractive for large scale commercial application.


In this article, we propose an  open loop approach which eliminates the problem of unwanted synchronization of TCLs without a need for two-way communication opening a path for implementing control functions with a heterogeneous population of TCLs using existing technologies. We show that by adding a small amount of intelligence (in the form of a few new elementary instructions) to the endpoint TCLs, we are able to implement simple open-loop controls which do not result in subsequent TCL load synchronization.  Such strategies to control large heterogeneous TCL populations, which we will call the safe protocols (SPs), could be quickly implemented to compensate for a generation-load mismatch without the fear of potentially dangerous subsequent TCL behavior.  In \cite{acc12}, one such a SP  for shifting temperature set point of the population of loads, without initiating any unwanted power oscillations, was proposed. The mechanism was shown to provide effective control on a relatively slow time scale (e.g. $\sim$30 minutes).  Here we extend the idea and show that embedded instructions allow to ``safely'' generate a class of power pulses of several basic shapes.  For example, we will show that it is possible to generate a sharp (e.g. 3 minutes long) power pulse with subsequent slow return of all TCLs to the initial customer set values of the temperature band. At the other extreme, we demonstrate another safe protocol that generates time-extended pulses ($\sim$0.5-1 hour) with changes in the TCL controlled temperatures of below 1$^o$C.  Such pulses are useful in spinning reserve applications.  In this work, we focus on the creation of these fundamental shapes and we postpone   the discussion of optimal pulse reconstruction and state estimation for future studies.

Our article is arranged as follows: In Section~\ref{sec:prob}, we will outline the major problems that appear with straightforward (not safe) strategies to control the power of TCLs; in Section~\ref{sec:SP1}, a demonstration of a SP is presented. In Section~\ref{sec:SP2}, we extend the idea and present a second protocol that achieves a shift of temperature set points of TCLs; Section~\ref{sec:SP3} presents a SP that creates a sharp power pulse of short duration; while Section~\ref{sec:SP4} briefly reviews possible further strategies to combine different SPs to achieve a desired power response. We summarize our results in Section~\ref{sec:concl}.

\section{Problem Description}\label{sec:prob}
Before we proceed with introduction to SPs, we outline the basic problems encountered when more straightforward ``unsafe'' approaches are applied to control the power consumption of heterogeneous TCL populations.

\begin{figure*}[htp]\label{}
\centering
\subfigure[{Dynamics of temperature ({\it top}) and consumed power ({\it bottom}) of a thermostatic load}]{\includegraphics[width=2.5in]{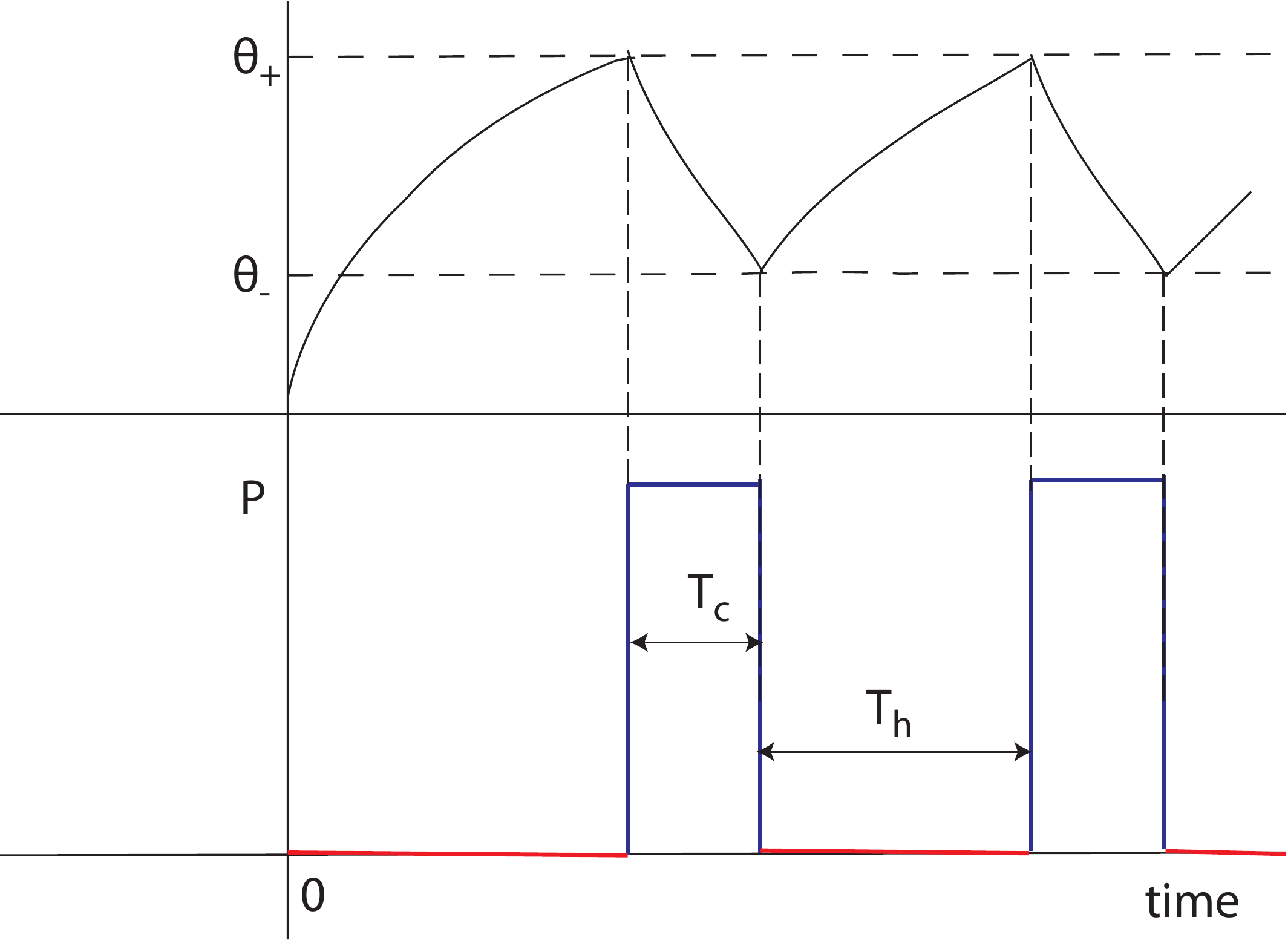}\label{fig:dynamics}}\hspace{0.01in}
\subfigure[{TCL probability  distribution  function (PDF)}]{\includegraphics[width=2.65in]{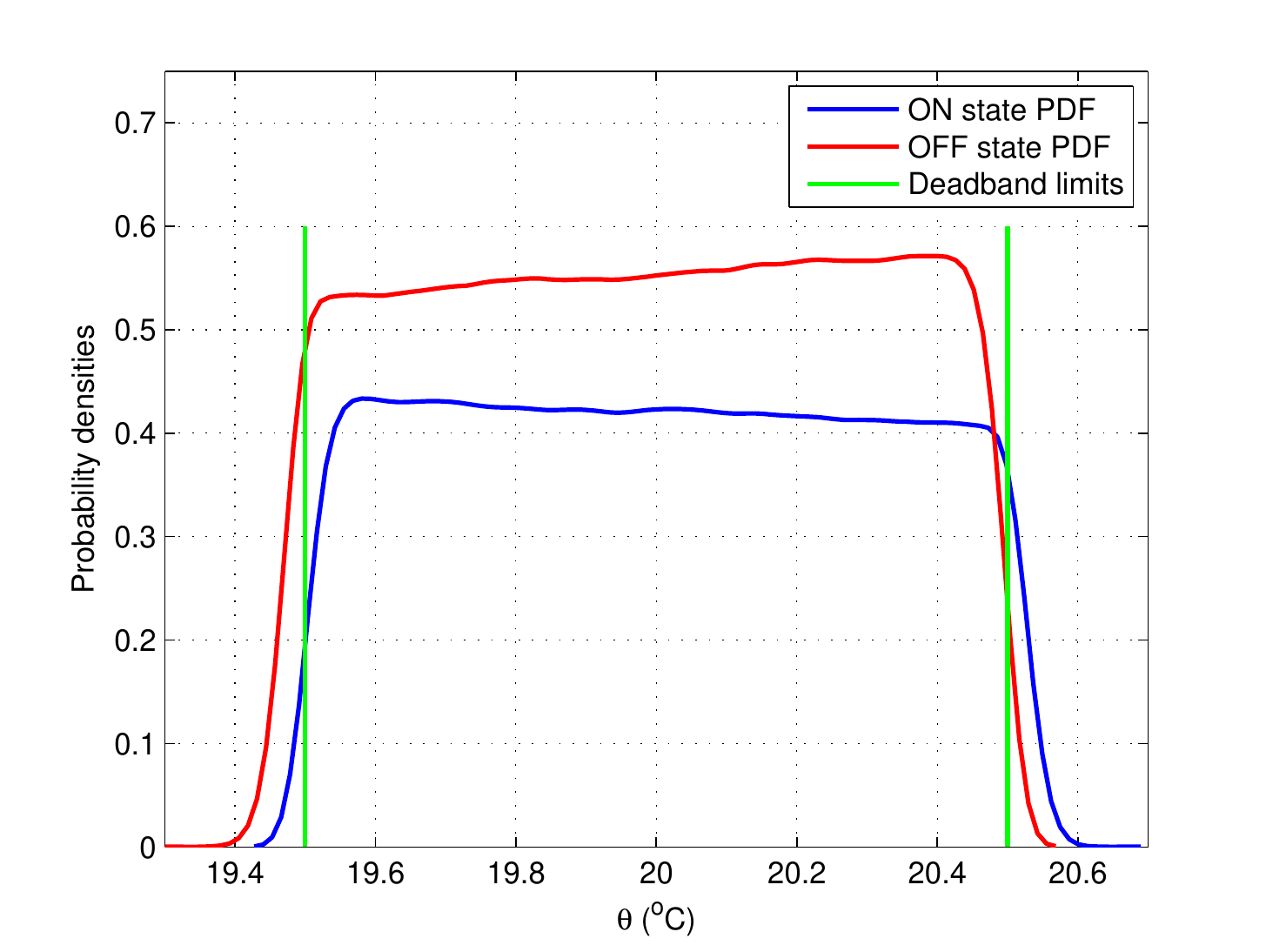}\label{fig:f1f0}}
\caption{{(left) Individual dynamics   and (right) aggregate steady state probability distribution  of TCLs over temperature in the OFF (red) and ON (blue) modes. Small nonzero density of TCLs beyond the deadband limits in (b) is due to the combined effect of temperature fluctuations and a finite response time to temperature measurements.}} 
\end{figure*}

When not subject to major thermal disturbances, a TCL's power consumption is cyclic in nature.  If we consider a TCL providing cooling, the controlled temperature $\theta$ drops when the TCL draws power (in the ON state) and increases when the TCL is not drawing power (in the OFF state). Evolution of temperature in a single house is generally modeled by a differential equation \cite{pscc}:
\begin{equation}
\dot{\theta} = \left\lbrace \begin{array}{ll} -\frac{1}{CR}\left(
  \theta - \theta_{amb} +PR \right), & \text{ON state} \\ &
  \\ -\frac{1}{CR}\left( \theta - \theta_{amb} \right), &
  \text{OFF state}	\end{array} \right. \label{eq:micro}
\end{equation}
\begin{figure}[htp]\label{fig:legends}
\centering
\includegraphics[width=2.65in]{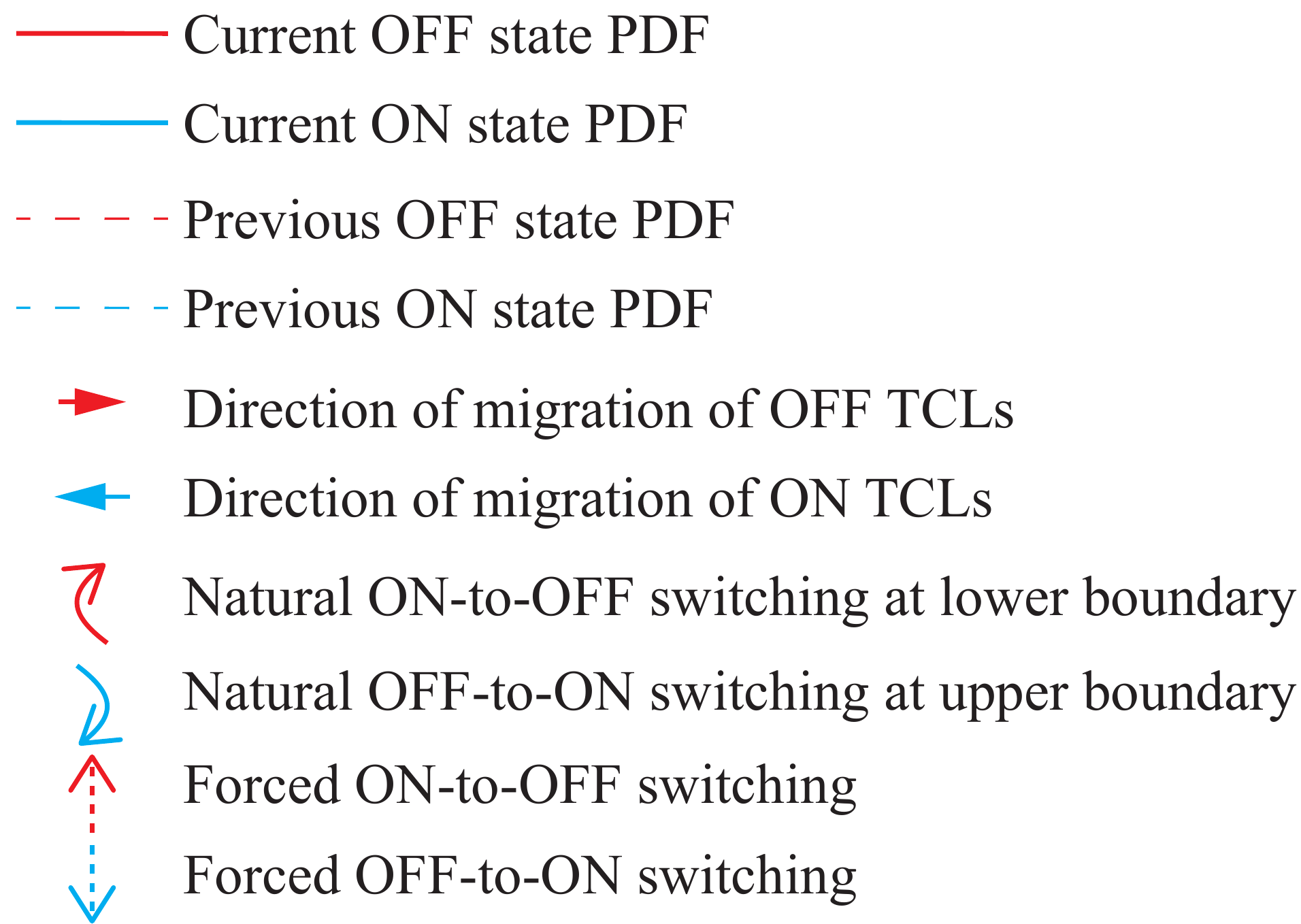}\label{fig:legends}
\caption{{A list of different signs used in the following figures.}} 
\end{figure}
where $\theta_{amb}$ is the ambient temperature, $C$ is the thermal capacitance, $R$ is the thermal resistance, and $P$ is the cooling power of the TCL when  in the ON state.  $P$  is proportional to the power drown from the Power Grid  with a coefficient that describes the efficiently of the TCL. We assume the regime of TCL operation with a duty cycle less than 100\%. Such a TCL switches its state from OFF to ON when its temperature increases to $\theta_+$ and from OFF to ON when temperature drops to $\theta_-$. Such dynamics and control push the TCL into a hysteresis around the set point temperature $\theta_s=\left(\theta_-+\theta_+\right)/2$ with a dead band of $\Delta=\left(\theta_+-\theta_-\right)$. Solution of Eqs. (\ref{eq:micro}) with such boundary conditions is shown in Fig.~\ref{fig:dynamics}~(black curve). Fig.~\ref{fig:dynamics}~(red and blue curves) shows that at steady state the power consumption of a TCL switches between a constant value and zero in ON and OFF states, respectively. The heterogeneity and independence of the TCL dynamics ensure that the time phase of  these oscillations are not correlated and do not result in noticeable aggregate power oscillations. If we start from a set of arbitrary TCL initial states, the heterogeneity of $C,R$ and $P$ values across the population and natural temperature fluctuations will cause the aggregate power demand of a large  ensemble of TCLs  settle down to a constant value with fluctuations due to the finite size of the population.  In this steady state, temperatures of all the TCLs would lie within their individual hysteresis dead bands and the temperatures of the TCLs attain steady state probability distribution \cite{callaway}{, as presented in Fig.~\ref{fig:f1f0}}. For simplicity, we will use this steady state as our reference state for demonstration of response of the population to control signals. { To help understand the figures to follow, a list of the different marks/signs used is presented in Fig.~\ref{fig:legends}.}

\begin{figure}[htp]
\centering
\includegraphics[height=3.5in]{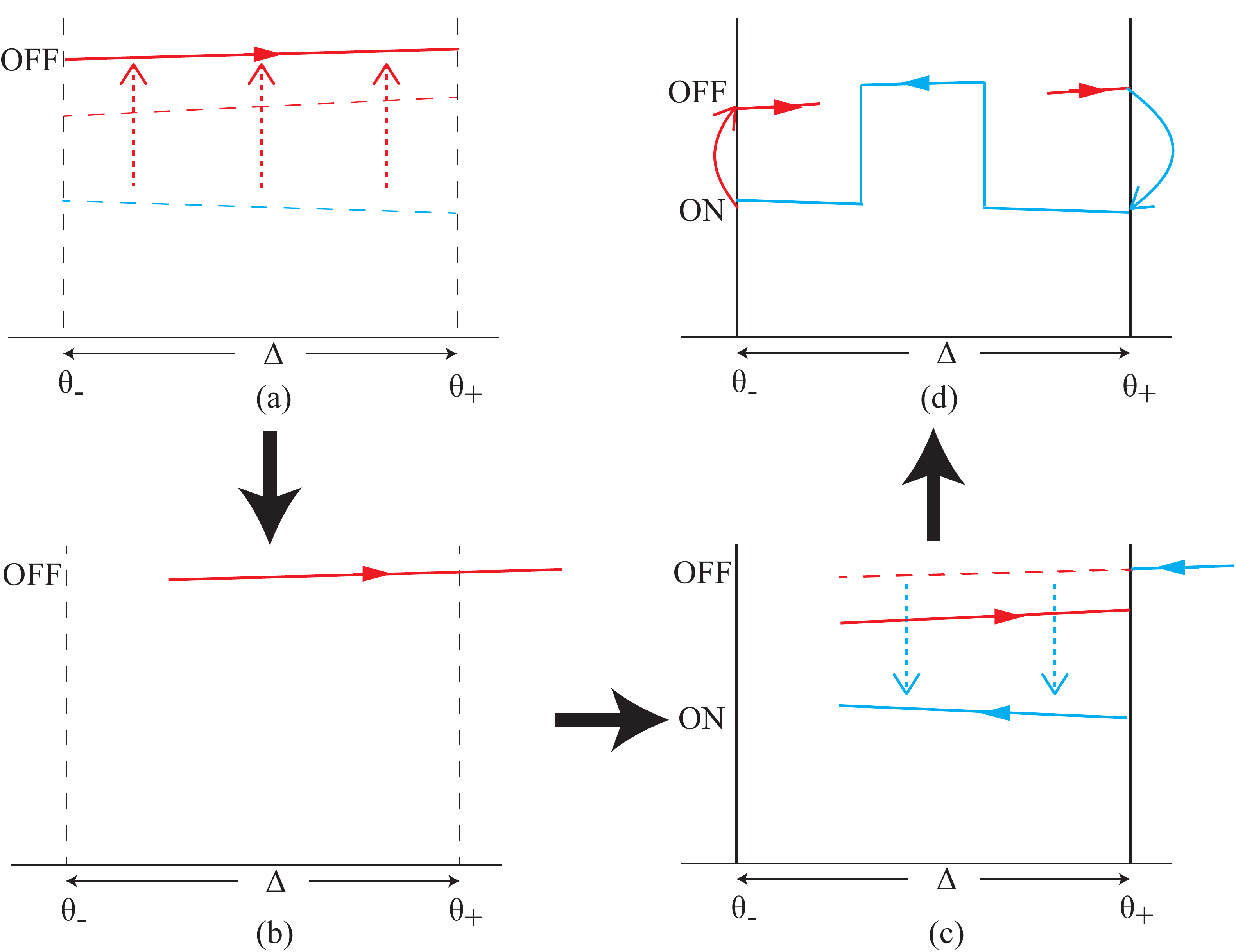}
\caption{Evolution of the TCL population distribution (y-axis) over temperature (x-axis) when all TCLs are are forced to stay OFF for a certain duration: (a) On receiving signal, all TCLs switch to OFF state. (b) The population distribution simply keeps moving to the right, crossing the upper deadband limit, with no natural switching at either of the deadband limits. (c) At the end of the desired duration, a signal is sent to TCLs to return to their original state which results, in particular, in all TCLs with $\theta>\theta_{+}$ turning ON. (d) After all TCLs return inside the originally set temperature range, a kink (higher than normal density of TCLs) and a hole (zero density over temperature) circulate with time.
Notation is explained in Fig.~\ref{fig:legends}}
 \label{fig:unsafe_switch-total}
\end{figure}
\begin{figure*}[htp]
\centering
\subfigure[{TCLs switched OFF for 10 minutes}]{\includegraphics[width=2.65in]{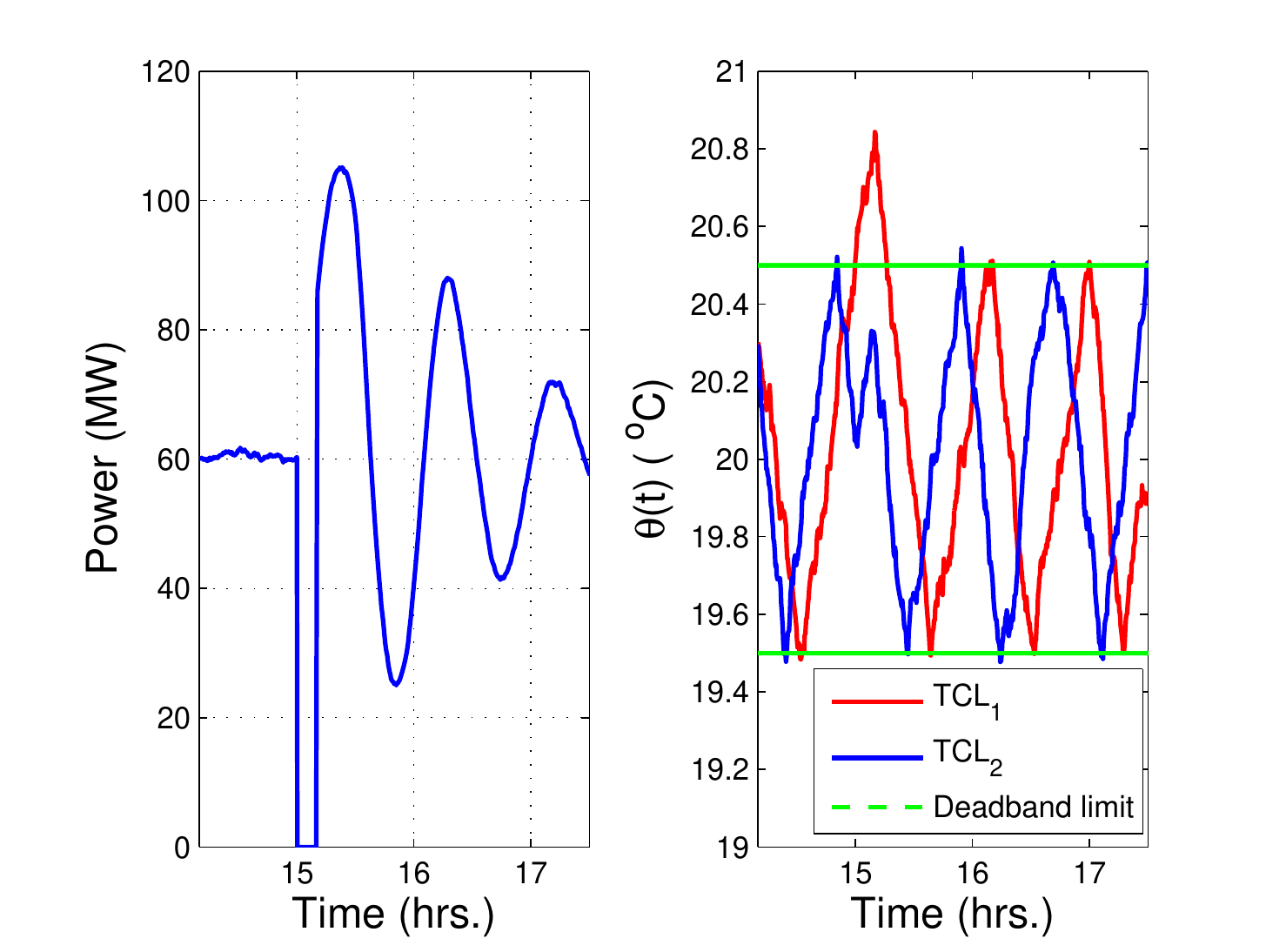}\label{fig:unsafe_sw_OFF}}\hspace{0.01in}
\subfigure[{TCLs switched ON for 10 minutes}]{\includegraphics[width=2.65in]{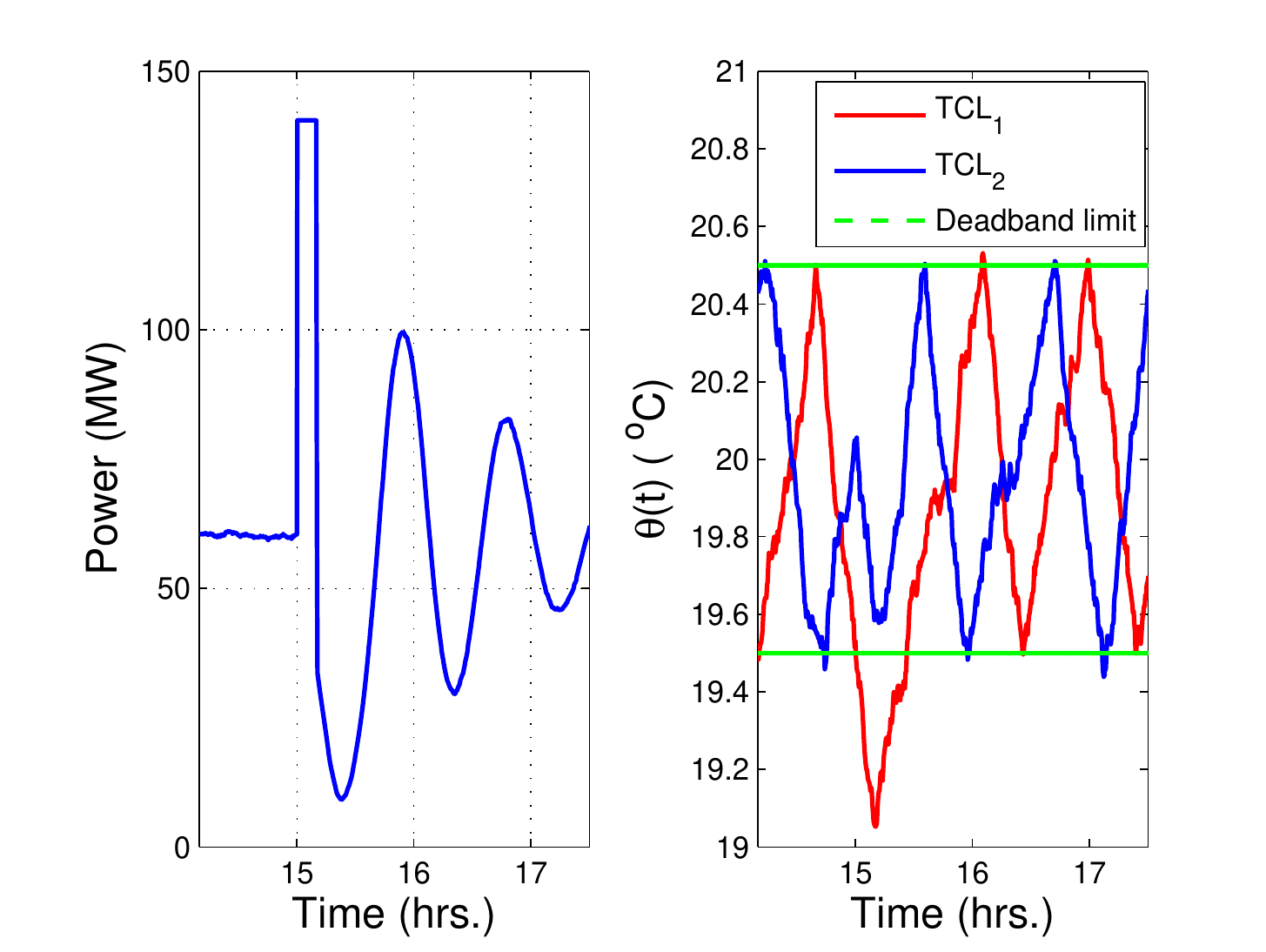}\label{fig:unsafe_sw_ON}}
\caption[]{(a)(left) Aggregate power consumption when all TCLs are forced to switch OFF for 10~min at 15 hrs. (a)(right) Red and blue curves show temperature evolution of two individual TCLs that are OFF and ON, respectively, at 15 hrs. (b) Same as in (a) but for the case when initially all TCLs are requested to switch ON for 10~min.}
\label{fig:unsafe_sw}
\end{figure*}
While undisturbed, the TCL dynamics remain uncorrelated, however, some forms of control can result in potentially dangerous correlations. For example, suppose that our goal is to reduce power consumed by all TCLs during a ten minute period and then return all TCLs to their individual dead bands.  Naively, one may try switching off all TCLs for 10 minutes and then sending a signal to all of them to start operating according to the initial conditions  \cite{kirby08,kirby07,heffner,kirby}. {Fig.~\ref{fig:unsafe_switch-total} shows that such a protocol would create a non-steady-state distribution with ``kinks'' and ``gaps'' in the ON-and OFF-state distributions.}

{In Fig.~\ref{fig:unsafe_sw_OFF} we show our simulations of behavior of a population of 10000 TCLs\footnote{We set $R$, $P$ and $C$  to follow the lognormal distribution in the ensemble of TCLs, which was argued to be a reasonable approximation to a real situation \cite{callaway,perfumo}, and we chose them  to achieve a typical TCL cycle duration of approximately $45$ minutes with a typical duty cycle approximately $40\%$ at $~\theta_{amb} = 32~^oC$ and $\theta_s=20~^oC$. Number of TCLs is $N=10,000$; 
 width of the deadband $\Delta=1~^oC$;  
Parameters $C,R$ and $P$ have mean values $3~kW/^oC, 2 ~^oC/kW$ and $14~kW$, respectively, and a standard deviation of $\sigma_p=0.07$ of the corresponding mean; a zero-mean Gaussian noise with standard deviation of $0.052^oC/min^{0.5}$ was added to
(\ref{eq:micro})  to account for natural temperature fluctuations. } in response to such a control signal.  
At $t=15$ hours, all the ON-state TCLs are switched OFF and all the TCLs are kept in switched OFF state for 10~min.  The aggregate TCL power, in Fig.~\ref{fig:unsafe_sw_OFF}(left), drops to zero, however, after receiving the signal to return to the initial state, the aggregate power does not return to the initial steady state value (60MW).  Instead, the population consumes much more power (with a peak above 100MW). During the ten minutes without the temperature control, many of the controlled temperatures increased beyond the customer upper dead band $\theta_+$, as shown in Fig.~\ref{fig:unsafe_sw_OFF}(right), and a sudden signal to return to the initial temperature range forces a larger fraction of TCLs to switch to the ON state than had originally switched to the OFF state. This phenomenon is  known as a {\it cold load pickup} \cite{ihara,chong84}. Similar synchronization happens when TCLs are switched ON for a 10-min interval (Fig.~\ref{fig:unsafe_sw_ON}).}

After providing beneficial control during the first ten minutes, this naive scheme creates an uncontrolled power spike of a similar magnitude and opposite sign. Perhaps even worse, the following power consumption relaxes to the steady state not monotonously but rather showing pronounced power oscillations, which may persist for several cycles depending on the population's heterogeneity. These oscillations, which we will call {\it parasitic oscillations}, are sustained by the synchronization of individual phases of TCLs induced by simultaneous transition of a large number of TCLs into the same ON-state.  Fig.~\ref{fig:unsafe_sw}~(left) shows the inverse situation when all the TCLs are set to be ON for ten minutes  after which the originally-ON-TCLs are switched back from OFF to ON resulting in a similar damped oscillation.

Another ``naive" approach to control the aggregated power of TCLs is to impose a sudden shift of dead band positions \cite{callaway,acc12}{, as shown in Fig.~\ref{fig:unsafe_shift-total}. 
\begin{figure}[htp]
\centering
\includegraphics[height=3.5in]{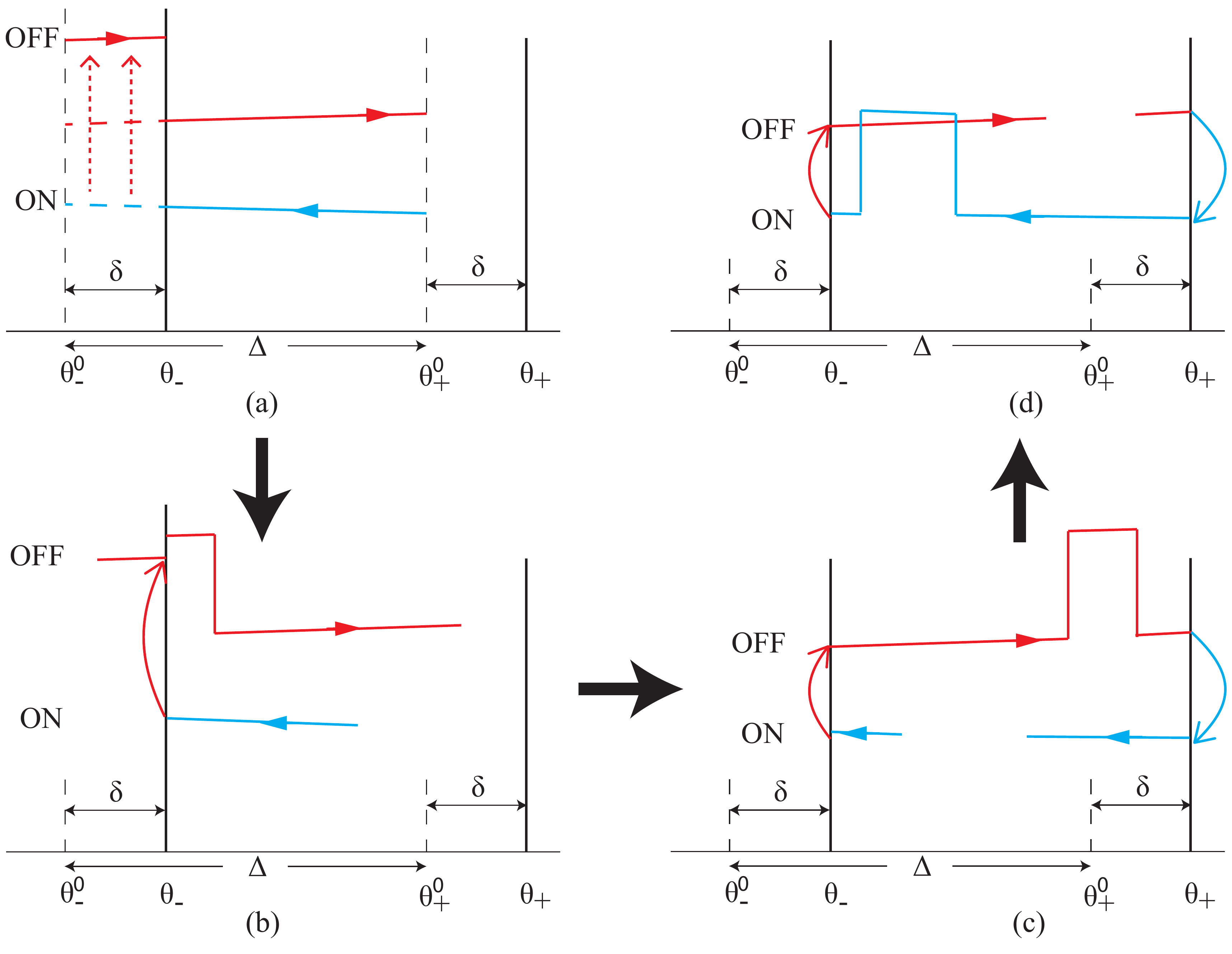}
\caption{Evolution of temperature distribution under sudden shift of deadband, by $\delta$, from $[\theta_{-}^0, \theta_{+}^0]$ to $[\theta_{-}, \theta_{+}]$: (a) As the deadband is shifted, a fraction of the initially OFF TCLs "instantly" switches to ON, creating a ``kink'' in the OFF-distribution. (b) While OFF-state TCLs are yet to reach the new upper deadband limit, ON-to-OFF switching already occurs at the new lower deadband limit. (c) As switching starts at both the new deadband limits, a ``hole'' is created in the ON-distribution. (d) The ``kink'' and ``hole' keep moving around the distribution.}
 \label{fig:unsafe_shift-total}
\end{figure}
\begin{figure*}[htp]
\centering
\subfigure[{Aggregate response under shift}]{\includegraphics[width=2.65in]{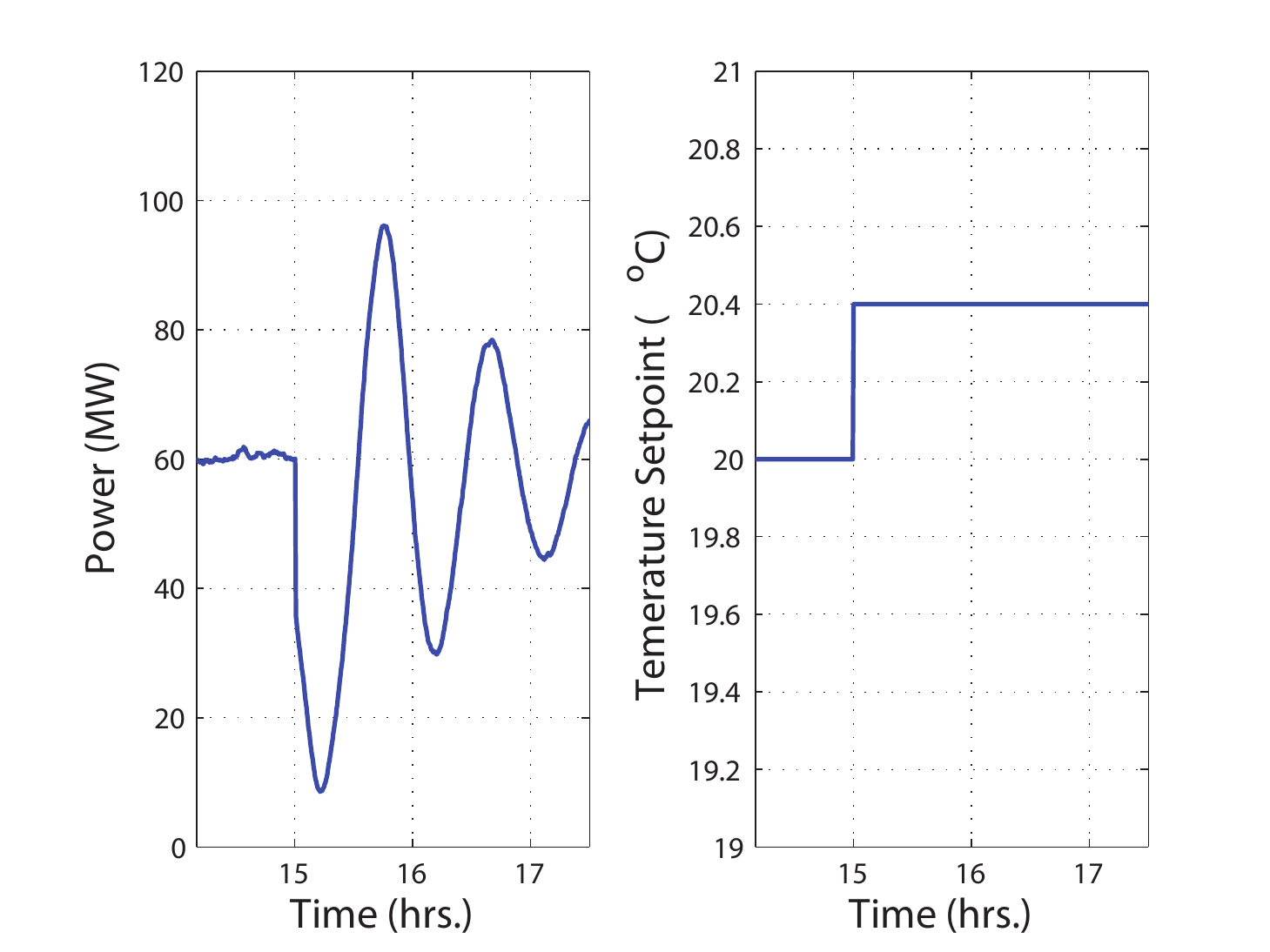}\label{fig:agg_unsafe_shift}}\hspace{0.01in}
\subfigure[{Typical response of individual TCLs}]{\includegraphics[width=2.65in]{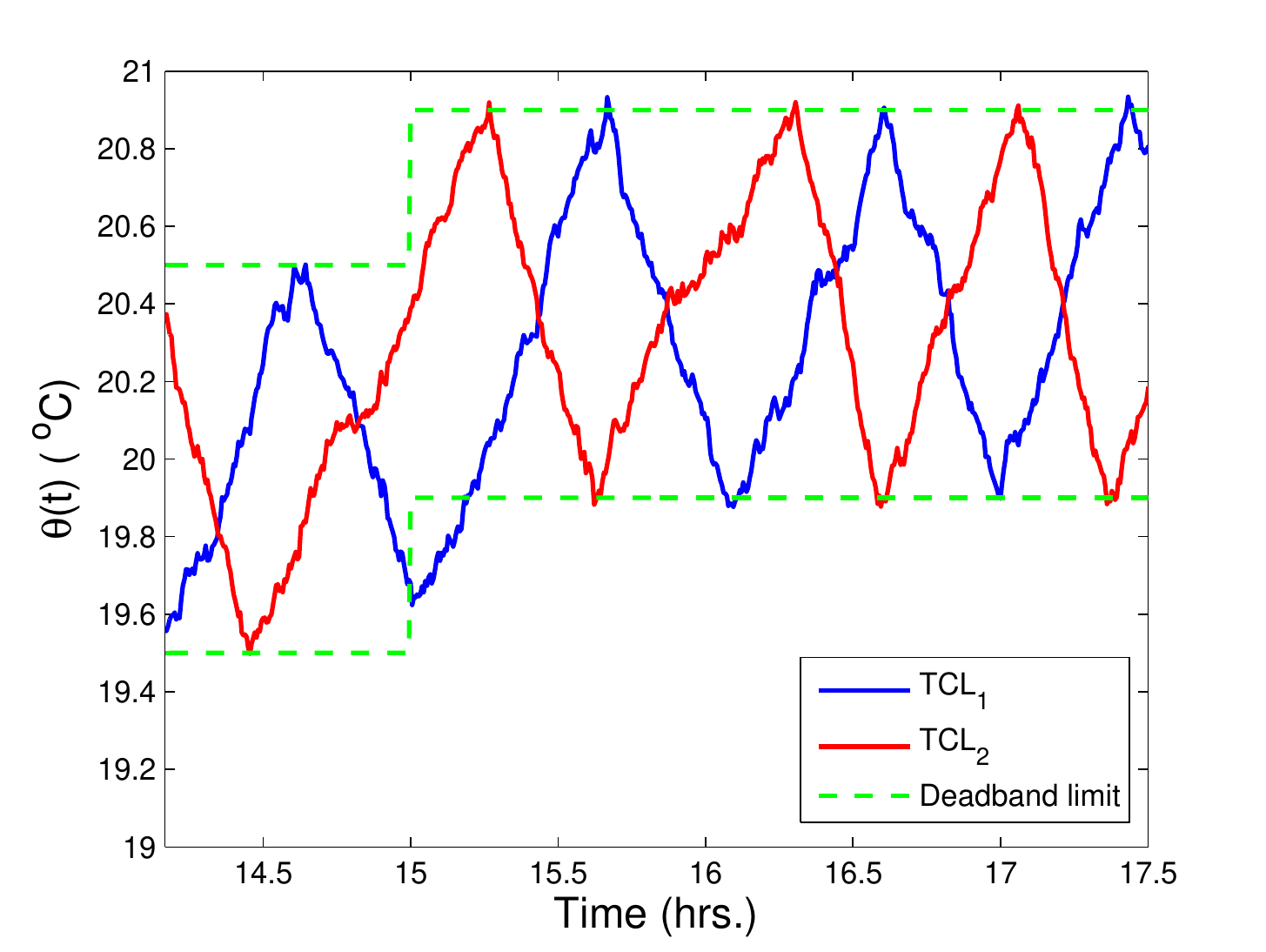}\label{fig:single_unsafe_shift}}
\caption[]{{(a)Aggregate response (left) under temperature setpoint shift (right)  by $0.4~^oC$. (b) Individual temperature evolution of two TCLs.}} 
\label{fig:unsafe_shift}
\end{figure*}
This approach of a sudden shift in deadband limits uniformly across the population also creates kinks and gaps in the temperature distribution, which sustain for a long duration. Fig.~\ref{fig:agg_unsafe_shift} shows our numerically calculated aggregated power response of a group of TCLs, which were initially operating in a steady state, after
their temperature set point was increased by $0.4~^oC$. As a result, some of the TCLs instantly found themselves lying below $\theta_-$, as shown in Fig.~\ref{fig:single_unsafe_shift}, and therefore switched OFF resulting in an immediate drop in the aggregate demand.} After this immediate drop, however, the aggregate demand undergoes large damped oscillatory mode due to essentially the same reasons as in previous example.

In summary, the synchronization of TCLs by external control signals may induce considerable power oscillations with amplitudes comparable to the total power consumed by the population at steady conditions. Control strategies that do not avoid such unwanted parasitic power oscillations may be more harmful than the power fluctuations that they attempt to remove.

\section{Safe protocol-1}\label{sec:SP1}
We start our discussion of safe strategies to create power pulses in the grid by considering a practically less interesting but more easy-to-follow example. Consider the strategy to control the population, which is illustrated in  Fig.~\ref{fig:safe2} and which we will term the safe protocol-1, or simply SP-1.  To implement any SP, an important assumption is that temperature controllers have additional embedded instructions. Specifically for SP-1, first, we assume that after receiving the control signal, TCLs are capable to save the information about the temperature that they are having at this moment. Second, we assume that TCLs are capable to compare their current temperature with their saved one. Finally, we assume that by reaching such a saved temperature, the TCLs can make transitions between ON and OFF states purely autonomously, i.e. without waiting for an external signal. Such additional instructions are elementary and can be easily added to temperature controllers of TCLs in most houses.

\begin{figure}[htp]
\centering
\includegraphics[width=3.7in]{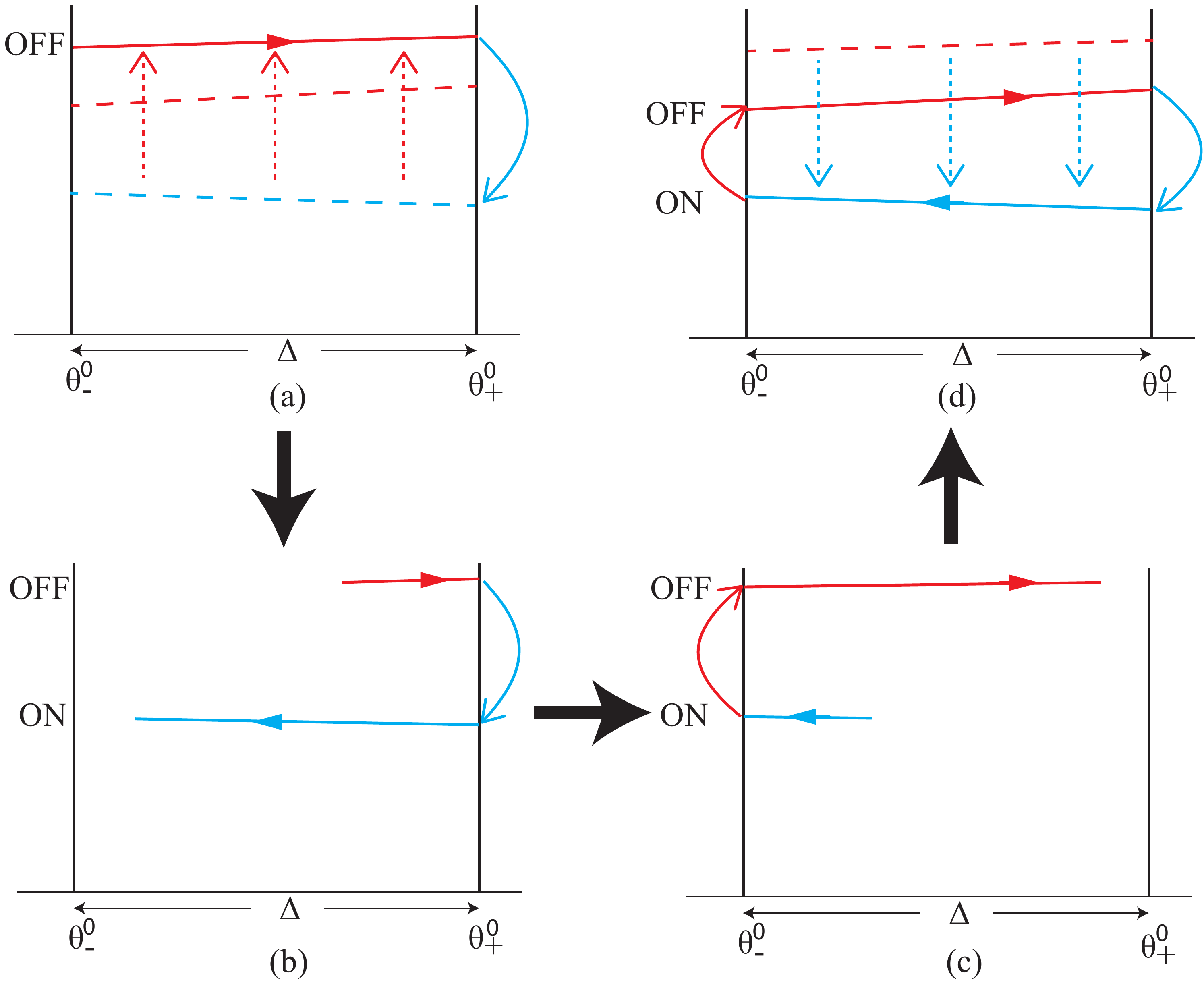}
\caption{Distribution of TCLs (y-axis) over temperature (x-axis) in the ON and OFF modes  during stages of safe protocol-1 (SP-1). (a) All TCLs are requested to switch OFF. (b) and (c) Intermediate stages of evolution of population distribution (y-axis is not to scale). (d) After all TCLs appear again in the OFF mode, TCLs that were in the ON mode before application of the protocol make autonomous transition from OFF to ON. As a result, 
the original steady state distribution is reproduced.} \label{fig:safe2}
\end{figure}

Having such additional capabilities, the generation of a safe pulse can proceed as follows:
{\begin{enumerate}
\item By receiving an external signal, TCLs from the ON branch switch to OFF (Fig.~\ref{fig:safe2}(a)).
\item  At this moment, TCLs that switched from ON to OFF save their temperatures.
\item All TCLs of the ensemble continue working without changing the position of the temperature operation band, as shown in Fig.~\ref{fig:safe2}(b)-(c).
\item When temperature of TCLs that switched OFF at step-1 reaches the same temperature as the saved one when it is again in the OFF-mode (this stage is achieved  after one period of the cycle, as shown in Fig.~\ref{fig:safe2}(d)), these TCLs autonomously switch to the ON state.
\end{enumerate}}
A similar process works if the power needs to be increased at initial moment, by switching TCL's states, initially, from OFF to ON at step-1.

{Fig.~\ref{fig:switch} shows two simulation results of application of SP-1 to instantaneously decrease (Fig.~\ref{fig:switch_off}) and increase (Fig.~\ref{fig:switch_on}) power in a heterogeneous population of TCLs. While after the immediate change, the aggregate power recovers by making a swing in the opposite direction, it quickly settles down back to its original steady state value after one cycle duration of a typical TCL in the population.
This becomes possible because the original distribution is reproduced after completion of stages of SP-1 (Fig.~\ref{fig:safe2}(d)). Fig.~\ref{fig:switch_off}(right) illustrates that, while TCL$_1$ that was OFF at the start of SP-1 at time 15 hrs continues with its natural cycle, TCL$_2$ switched its state from ON to OFF and then switched back from OFF to ON after a duration equal to its natural cycle period. An equivalent behavior occurs when TCLs are initially switched ON, Fig.~\ref{fig:switch_on}.}

\begin{figure*}[htp]
\centering
\subfigure[{TCLs switched OFF}]{\includegraphics[width=2.65in]{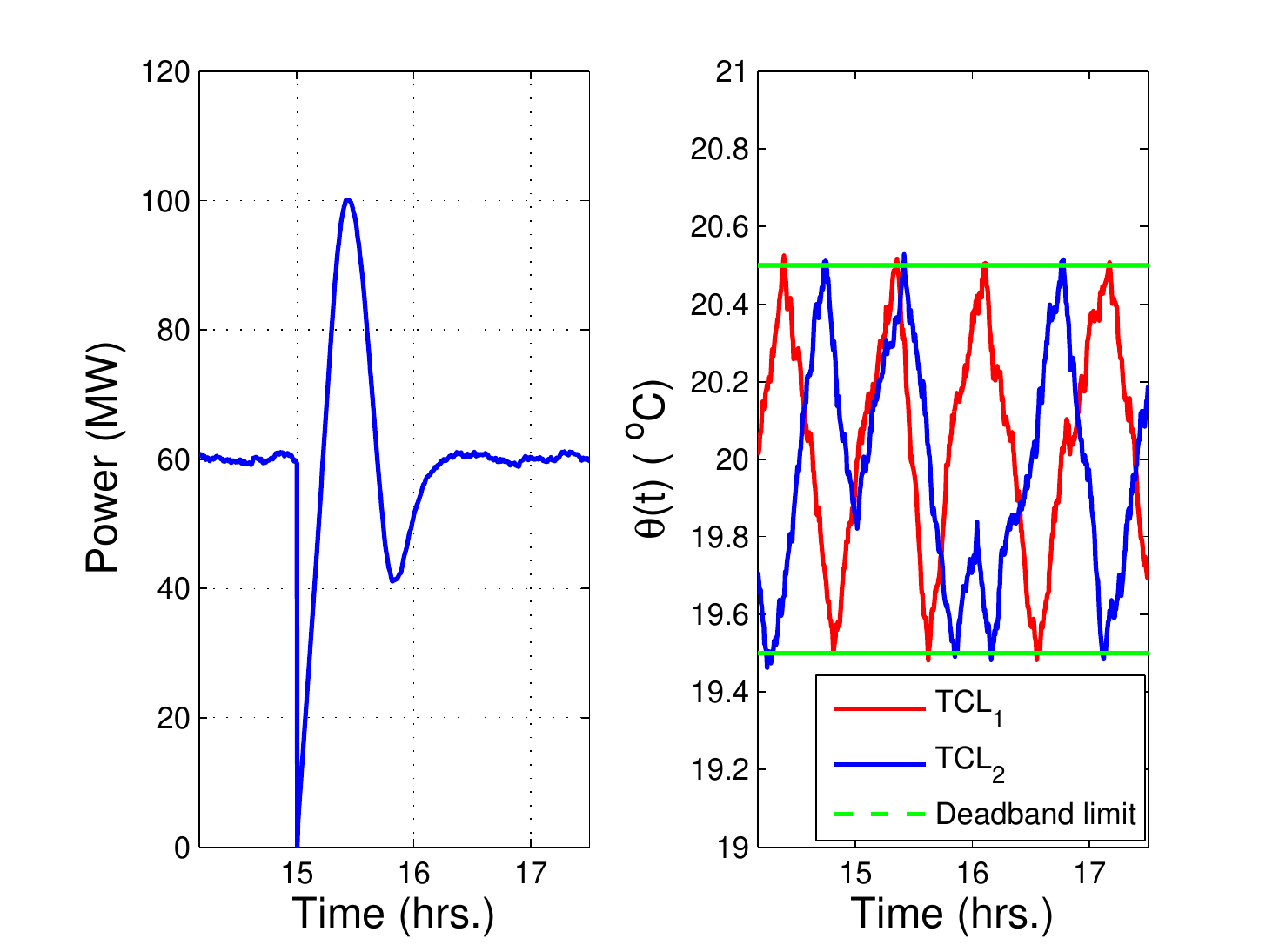}\label{fig:switch_off}}\hspace{0.01in}
\subfigure[{TCLs switched ON}]{\includegraphics[width=2.65in]{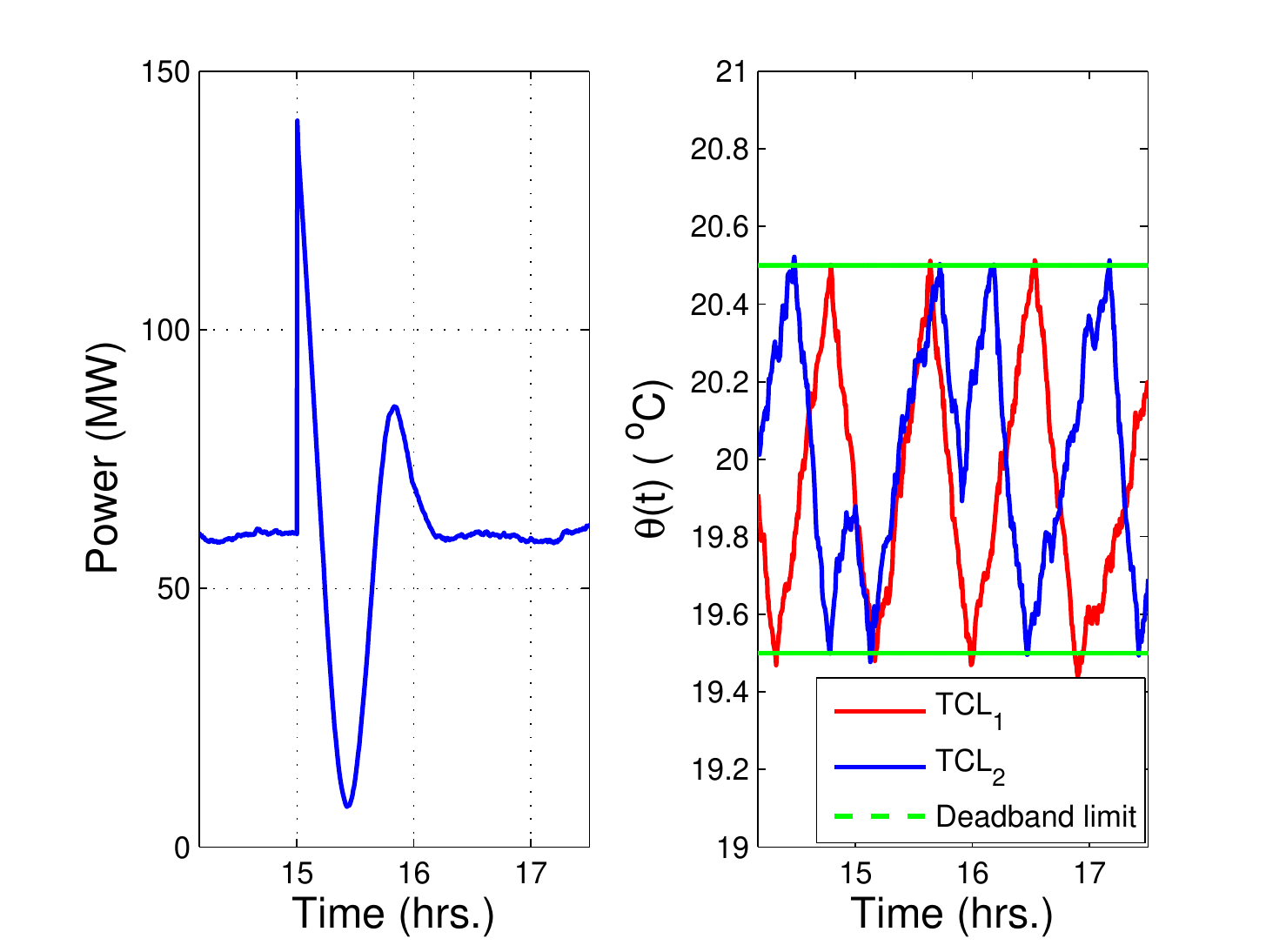}\label{fig:switch_on}}
\caption[]{{Aggregate power (left) and individual temperature (right) response of TCLs under SP-1. Notation is explained in caption for Fig.~\ref{fig:unsafe_sw}.}}
\label{fig:switch}
\end{figure*}
%

Here we note that Fig.~\ref{fig:safe2} shows that the SP-1 ends with equilibrium distribution of TCLs over temperature if the population is homogeneous, i.e. if all parameters of TCLs are identical with identical cycle duration. A real population would be certainly heterogeneous. Nevertheless, Fig.~\ref{fig:switch} for  the power response shows that SP-1 does not lead to unwanted power oscillations by the end of the protocol even if it is applied to a {\it large} heterogeneous population described in the footnote in Section II. The reason is that for any TCL in a large  heterogeneous population,  there are many other TCLs with similar parameters. That is, each large population can be considered as a set of many almost homogeneous sub-populations. Since a SP does not lead to unwanted power oscillations in {\it any} homogeneous sub-population, it  does not induce such oscillations when it is applied to a full heterogeneous ensemble of TCLs either.

Safe protocol SP-1 has two additional useful features:
\begin{itemize}
\item Power increases almost instantly in response to the external signal so that the protocol can be used to offset  fast power fluctuations. In reality, of course, there is always a time delay between receiving the signal and changing the mode of operation of a TCL, which restricts the speed of control.  Response time of a realistic TCL to external control, however, is as small as a fraction of a minute \cite{kirby}.
\item Since the temperature deadband remains undisturbed, this protocol guarantees that we never leave the temperature operation band that was set by a customer. Hence, application of this protocol should be practically unnoticeable by owners of TCLs.
\item Finally, and this will be a generic property of all safe protocols discussed in this article, the safe protocols are also safe from the privacy point of view. They do not require the utility company to have information about any specific TCL. Only aggregate information about the total power consumed by the ensemble is needed to estimate the power response to a SP.
\end{itemize}

However, the applicability of SP-1 alone  is limited because SP-1 cannot generate any extra energy by the time of its completion in comparison to the energy generated by the same ensemble at the steady state during the same time. At the end of SP-1, the ensemble of TCLs is brought to the initial distribution, exactly the same state in which ensemble would be if no external control is applied. Intuitively, the conservation of energy argument suggests that SP-1, while producing a nontrivial power pulse, does not save any extra energy by the time of its completion, which is indeed the case as we prove in the Appendix.

\section{Safe Protocol-2}\label{sec:SP2}

The next protocol, SP-2, shows that the drawback of the previous protocol, i.e. the inability to store/release energy by the time of its completion, can be eliminated by choosing a different strategy. In particular, SP-2 creates a  power pulse with a net release of energy, in comparison to the undisturbed ensemble, to the grid. This is achieved during a shift of the positions of TCL's temperature operation bands by a fixed amount of temperature.
This protocol has been already presented previously \cite{acc12}, so we only describe it here briefly because it will be used later in our discussion.
\begin{myindentpar}{0.3cm}
To describe the SP-2, we should introduce the {\it transition points} that are {\it intermediate} hysteresis deadband limits that become operative starting from the instant the temperature shift is initiated until the new set of deadband limits are applied across the population. Let $\theta_-^0$ and $\theta_+^0$ be the lower and upper deadband limits before a shift of the deadband position to the new limits  $\theta_-=\left(\theta_-^0+\delta\right)$ and $\theta_+=\left(\theta_+^0+\delta\right)$. If the temperature setpoint is to be increased by an amount $\delta$, then the transition points would be $\theta_-^0$ and $\left(\theta_+^0+\delta\right)$. While for a decrease in setpoint by an amount $\delta$, the transition points would be $\left(\theta_-^0 - \delta\right)$ and $\theta_+^0$. Note that positions of the transition temperature points do not coincide with the positions of the new dead band limits.
\end{myindentpar}
The SP-2 was constructed as follows:
 \begin{enumerate}
 \item A TCL continues to stay in its present state (ON or OFF) until it hits one of the transition points.
 \item Once it has reached one of the transition points, the new pair of deadband limits, $\theta_-=\left(\theta_-^0+\delta\right)$ and $\theta_+=\left(\theta_+^0+\delta\right)$, starts governing its thermal dynamics.
 \end{enumerate}
Here $\delta>0$ for the increase in the setpoint and $\delta<0$ for the decrease in the setpoint. Fig.~\ref{fig:safe1} shows how this protocol works when the setpoint temperature is increased by $\delta$. The aggregate power demand of the population of TCLs attains a new steady state corresponding to the shifted deadband positions within a thermostat's cycle duration. Fig.~\ref{fig:safe_increase} illustrates how SP-2 performs when the setpoint is increased.{The aggregate response in Fig.~\ref{fig:agg_safe_increase} shows that the power drops for a duration equal to the time period and then stabilizes to the new steady state value without any parasitic oscillation. Fig.~\ref{fig:single_safe_shift} shows two typical responses of individual TCLs. TCL$_2$ was initially OFF and hence simply continues with the new deadband limits, while TCL$_1$ which was ON at the initiation of SP-2 continues with its old deadband limits until it hits the lower deadband limt after which it adopts the new limits.}
\begin{figure}[htp]
\centering
\includegraphics[width=4in]{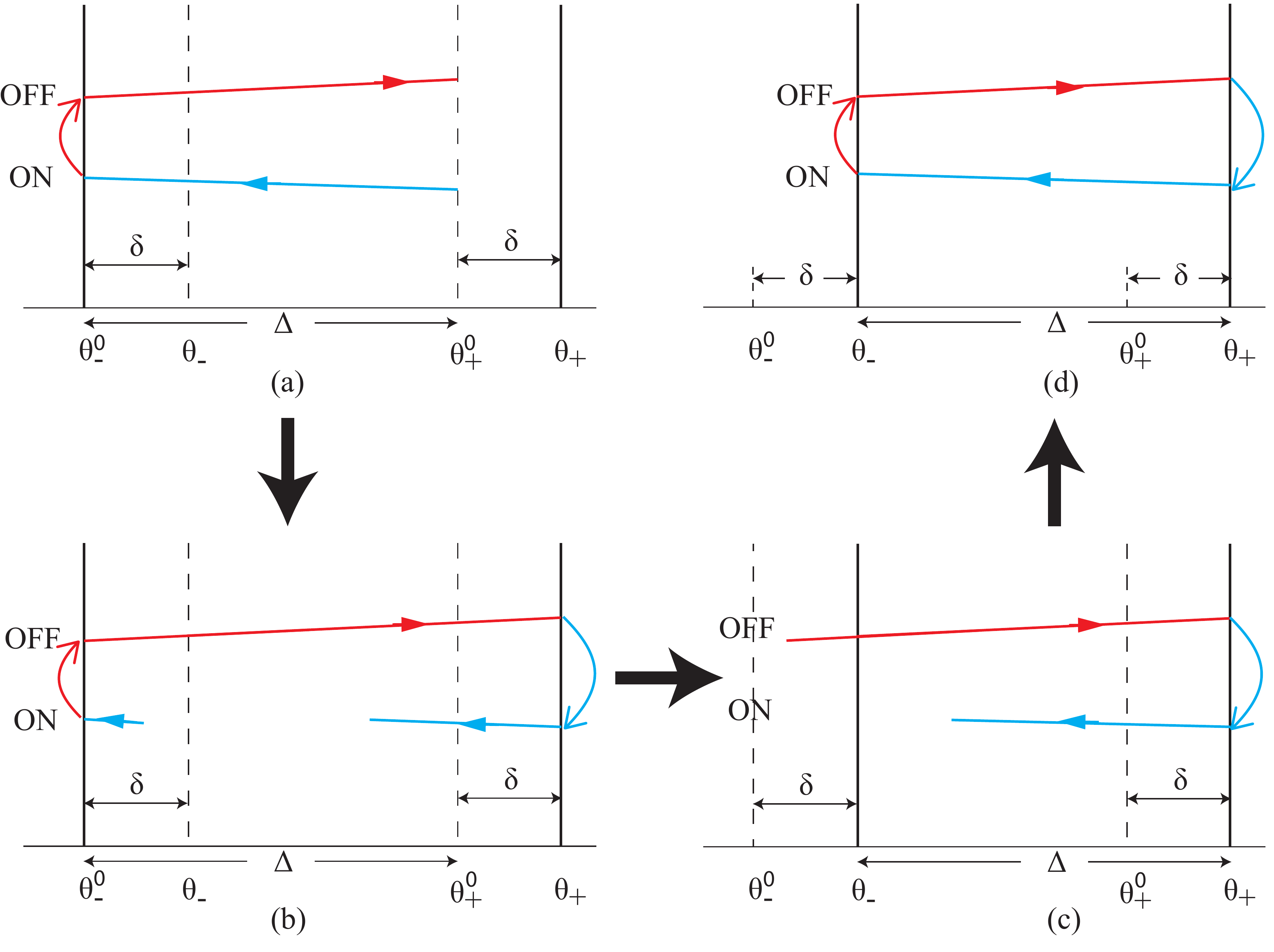}
\caption{Evolution of TCL's distribution over temperature in the ON and OFF modes during states of Safe protocol-2 (SP-2), which is based on shifting temperature range limits from initial, ($\theta_{-}^0, \theta_{+}^0)$, to  new ones, ($\theta_{-}, \theta_{+}$).} \label{fig:safe1}
\end{figure}

\begin{figure*}[htp]
\centering
\subfigure[{Aggregate response under SP-2}]{\includegraphics[width=2.65in]{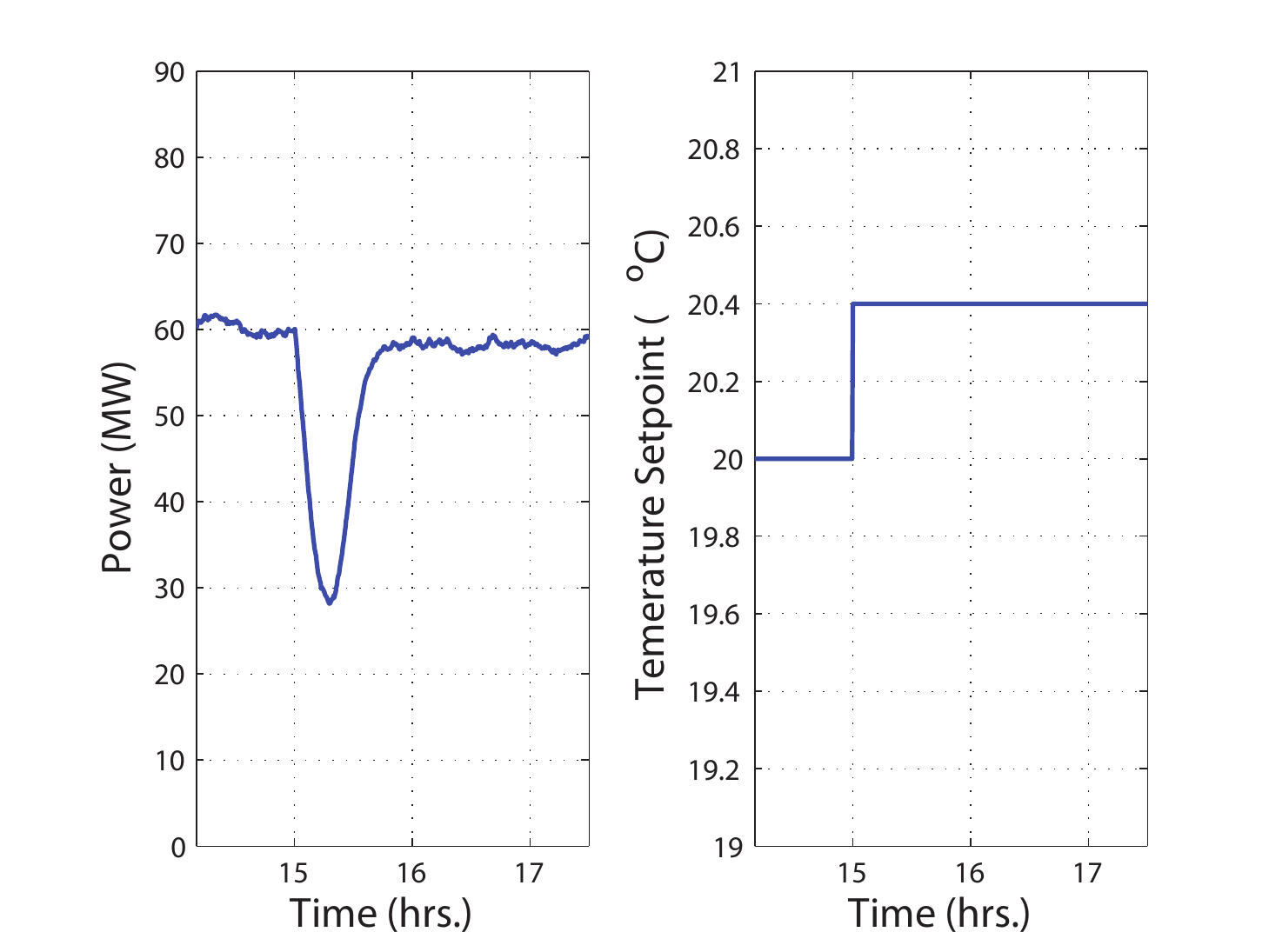}\label{fig:agg_safe_increase}}\hspace{0.01in}
\subfigure[{Typical response of individual TCLs}]{\includegraphics[width=2.65in]{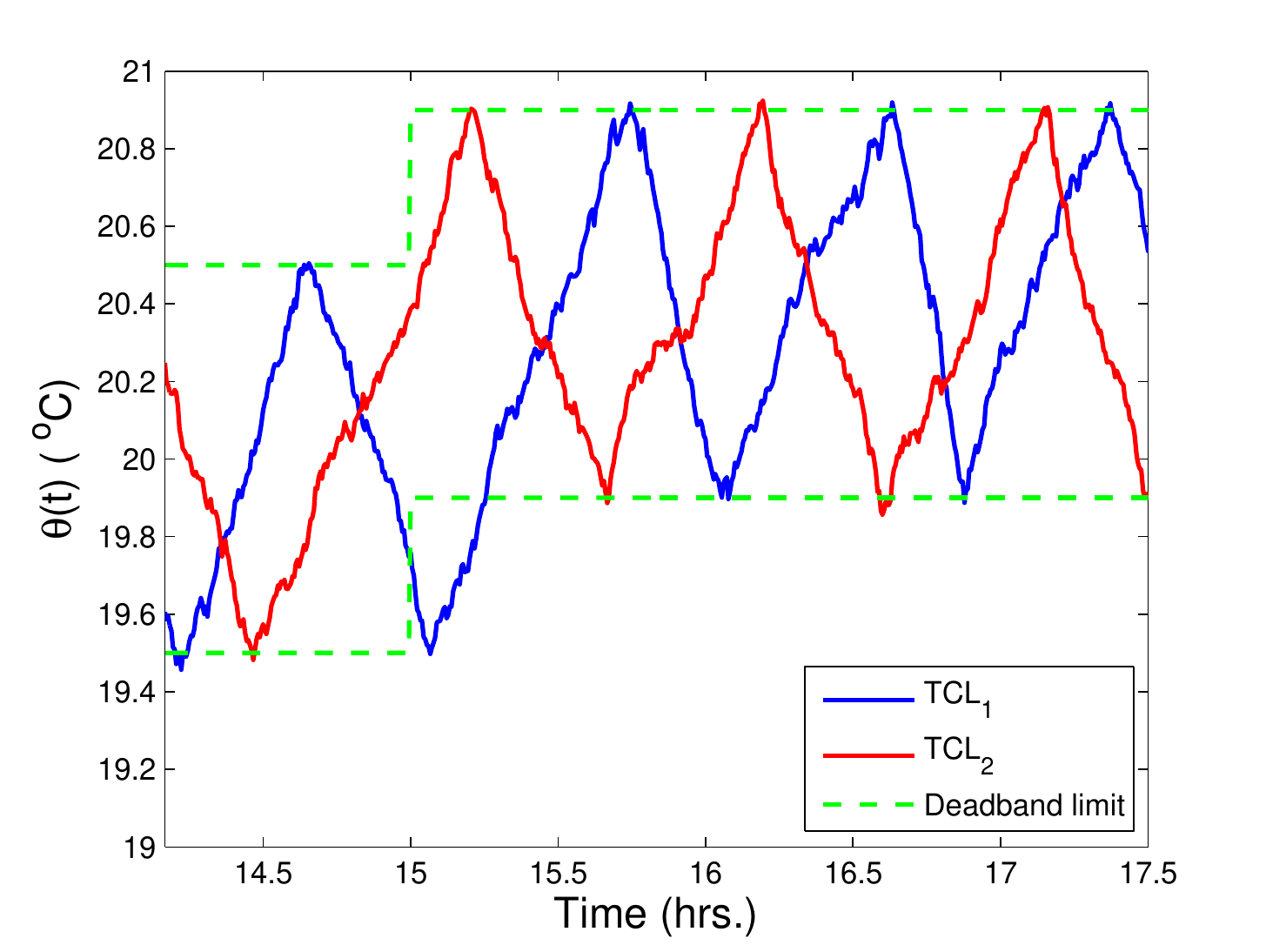}\label{fig:single_safe_shift}}
\caption[]{{Aggregate power and individual temperature response of TCLs under SP-2 when the temperature setpoint is increased by $0.4~^oC$. Notation is explained in caption for Fig.~\ref{fig:unsafe_shift}. }}
\label{fig:safe_increase}
\end{figure*}

Besides the absence of parasitic oscillations, this safe protocol (SP-2) has several new advantages:
\begin{itemize}
\item It can be used to switch the set point temperatures almost in arbitrary range, limited only by customer tolerance level.
\item This protocol practically does not require extra switching of the device branch per time of the driving protocol. The completion of the protocol takes about the cycle duration time and by this time all TCLs switch their mode (from OFF to ON or from ON to OFF) only one time. In comparison, SP-1 takes about the same time before its completion during which each TCL makes 4 mode switching events, while a steady state cycle includes only 2 switching events between OFF and ON states per one TCL cycle.This is important because more frequent switching of the TCL operation mode accelerates the fatigue of the device circuits and may be considered the major cost of application of the control.
\end{itemize}

By careful examination of Fig.~\ref{fig:safe1}, one can recognize the reason why SP-2 does not lead to unwanted power oscillation: All steps are chosen to avoid creation of "kinks"  or "holes" in the final distribution of TCLs over temperature.

The major drawback of  SP-2 is its ``slowness''. Power does not change abruptly after receiving a control command. It takes a certain while to reach the extremum of power consumption by an ensemble. Hence SP-2, alone, is restricted to very specific applications that do not require fast response, e.g. for offsetting relatively slow ($\sim$30-60 minutes in our case) but possibly large power fluctuations.
For example, one can envision a situation when a power line breaks down and one wants to decrease the burden by TCLs on the Power Grid during a relatively large time interval, e.g.  to provide time for a backup coal power plant to increase power output. Such situations are commonly resolved by relatively expensive, in comparison to standard power generators, spinning reserves, which annual cost on the scale of such a country as the USA is considerable. This justifies the importance of the SP-2 for helping to resolve such problems.

\section{Safe Protocols to Generate Short Power Pulses}\label{sec:SP3}

TCL's ability to respond to control signals at a fraction of a minute suggests that they may not be simply considered as replacement to standard spinning reserves. They can be  also used to increase stability of the power supply in the power grid by offsetting sharp power spikes. For such applications, it is desirable to design control protocols that are safe against synchronization problems but react much faster and create very short power pulses.
While safe protocols SP-1 and SP-2 could create power profiles without unwanted power oscillations, none of them is suitable for  offsetting short (2-5 minutes) power spikes because the pulses that they generate continue for a duration of a  cycle of a typical TCL in the population.

\begin{figure}[htp]
\centering
\includegraphics[width=4in]{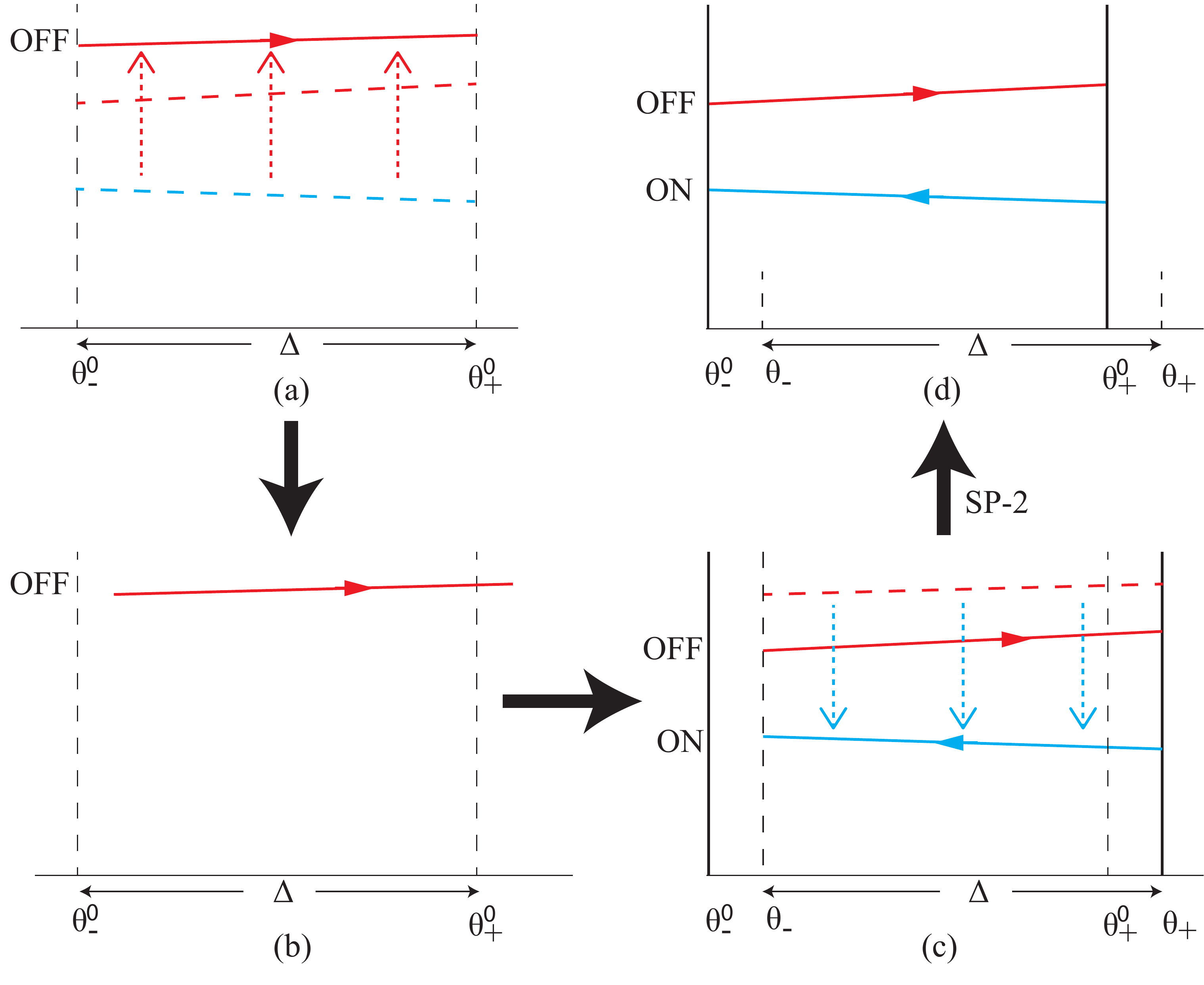}
\caption{SP-3 based on switching (to OFF) and shifting, followed by SP-2, to bring the population of TCL's to the initial setpoints.} 
\label{fig:safe3}
\end{figure}
Here we construct the protocol, referred to as  SP-3, that generates a short power pulse of possibly large amplitude. If we consider that time of response of a power plant to fluctuations is e.g. 8~minutes, then what one would need to increase stability of the power supply is to be able to create shorter sharp pulses of duration, e.g., $\Delta t=3$~minutes with a TCL ensemble and then return all TCLs to initial setpoints avoiding sharp pulses of opposite sign and power oscillations. This is achieved by the protocol, which stages are shown in Fig.~\ref{fig:safe3}.
The mechanism of SP-3, when the need is to decrease power, is as follows:
\begin{enumerate}
\item By receiving the signal, TCLs in the ON mode switch to OFF (Fig.~\ref{fig:safe3}(a)) and save information about their original deadband limits, $\theta_{-}^0$ and $\theta_{+}^0$.
\item {All TCLs continue on OFF mode for a duration $\Delta t$ (say, 3~min), as shown in Fig.~\ref{fig:safe3}(b). It is to be noted, that during this time $\Delta t$, no switching takes place at either of the deadband limits.
\item At the end of this interval, TCLs that were forced to switch the branch at step-1 now should switch from OFF to ON mode. The result is shown in Fig.~\ref{fig:safe3}(c).}

Note that at this stage no ``kinks'' or ``holes'' in the distribution  were created.
 Steps 1-3 will produce sharp drop of power consumption for a time $\Delta t$ and resulting distribution at the end of step-3 will be close to the new equilibrium because, after step-3, we did not create any kinks or holes in TCL's distribution. These steps, however, were also discussed in Section 2, in which we showed that if, following step 3, we simply ask all TCLs to return to initial conditions, parasitic power oscillations can be considerable. In order to fix this problem, {we propose an application of SP-2 to effect a slow return of deadband limits}. 

\item After step-3, the ensemble is brought to the initial state {(Fig.~\ref{fig:safe3}(d))} by application of a specific version of SP-2 (refer to Section~\ref{sec:SP2}). SP-2 that was described in the previous section requires information about the current deadband limits and the new deadband limits (or, alternatively the amount of shift in setpoint temperature), to use as transition points. But the ``current'' deadband temperature limits, at the end of step-3, $\theta_{-}$ and $\theta_{+}$ are not exactly known for a heterogeneous population. One can show, however, that this knowledge is not needed. We can simply instruct all TCLs that switch to ON-branch after step-3 to start operation according to the old, saved, temperature deadband ($\theta_-^0$ to $\theta_+^0$) after step-3. As for the TCLs that were in the OFF-branch before step-1, we  instruct them to continue staying on this branch for a duration $\Delta t$ from the moment they reach the temperature $\theta_{+}^0$ (saved in step-1). This can be achieved by adding timers to the temperature controllers so that the timer starts autonomously after reaching $\theta_{+}^0$. Only after completing such extra $\Delta t$ delay in the OFF-mode, these TCLs switch from OFF to ON and then set their thermostats to work according to the old band position. This set of instructions will not create kinks in the distribution and will be equivalent to application of SP-2, with effective transition point at $\theta_{-}^0$ and $\theta_+$, following step-3.
Note that step-4 is necessarily slow so that sharp and strong drop of power during steps 1-3 is compensated by a slow and weak increase of power during step-4.
\end{enumerate}

Fig.~\ref{fig:safe_sw_exact} shows simulation results of the application of SP-3. A negative pulse of width 3~min is generated following the steps 1-4 of SP-3, as shown in Fig.~\ref{fig:safe_sw_exact}~(right). The sharp and exact positive pulse is followed by a slow and weak change in aggregate power in the opposite direction for about a cycle duration before returning to the original steady state consumption level. A similar result, but in opposite direction, happens when a positive pulse of width 3~min is generated (Fig.~\ref{fig:safe_sw_exact}~(left)).
\begin{figure}[htp]
\centering
\includegraphics[width=2.65in]{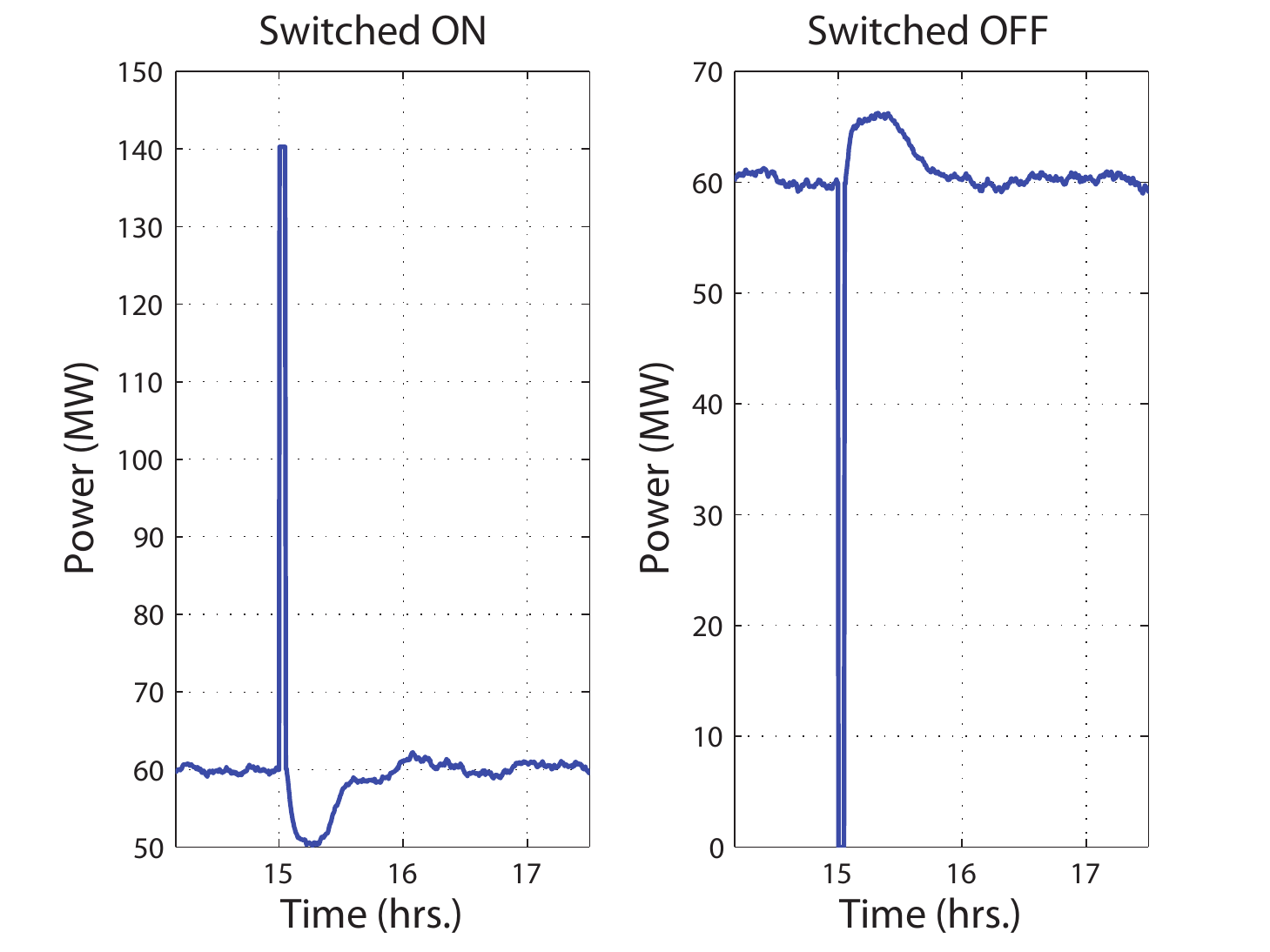}
\caption{Power response under SP-3 generating a positive pulse by switching ON (\textit{left}) or negative pulse by switching OFF (\textit{right}), both of a width of 3~min.}
\label{fig:safe_sw_exact}
\end{figure}

This method is particularly suitable for applications, in which the required pulse width is much smaller than the cycle duration, e.g. less than 5 minutes. This strategy opens the possibility of an application of TCLs to a class of problems that standard spinning reserves cannot resolve.


\section{Hybrid protocols}\label{sec:SP4}

In previous sections, we described three protocols that generate power pulses of different shapes with high degree of control. In real situations, those shapes separately may not be sufficient to achieve all utility goals - e.g. creating arbitrary shape of power response. The latter goal can be achieved by combining different pulses. For example, a train of power spikes can be eliminated by a train of SP-3 pulses sent to sub-populations of TCLs. However, applying SP-3 alone in a series of pulses may not be the best strategy for some goals. For example, if TCLs are used to replace standard spinning reserves, then the typical situation that they may need to resolve is to provide extra power supply at fixed level for relatively long duration, e.g. of 30 minutes. In this section, we show an example of a hybrid protocol that can be used in such a specific situation.

Suppose that in response to a damage of a power line, a spinning reserve should react almost instantly and supply extra power in the Power Grid during ~30 minutes.
In principle, extra power can be supplied by SP-2  but, alone, it is slow. On the other hand, SP-1 responds "instantly" but does not supply an overall extra power.
\begin{figure}[htp]
\centering
\includegraphics[height=4in]{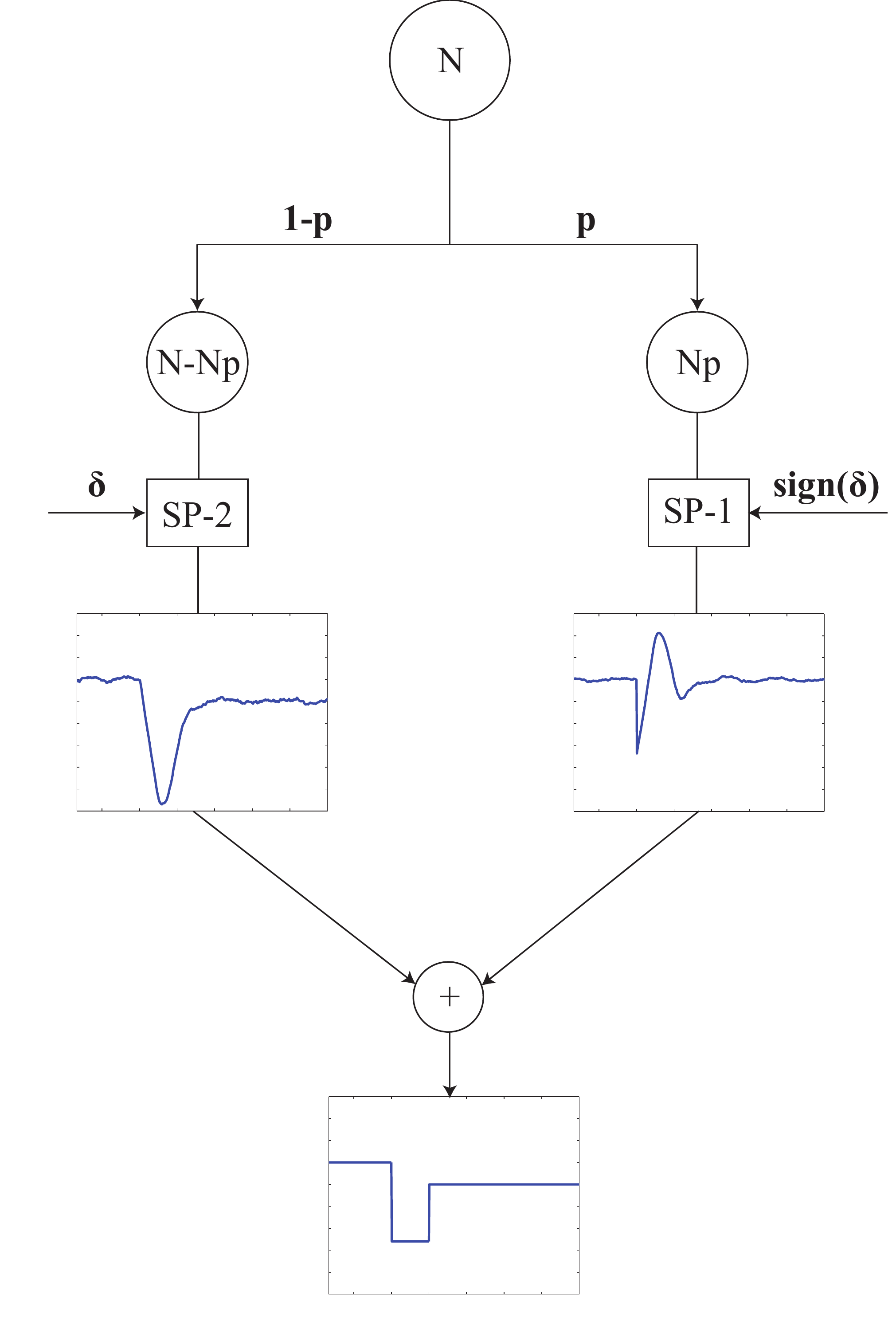}
\caption{{Combination of SP-1 and SP-2 into a hybrid protocol to generate a (negative) pulse. $N$ is the total number of TCLs in the population, $p$ is the fraction of TCLs that receive SP-1 signal and $(1-p)$ is the fraction of TCLs that receive SP-2 signal. }} \label{fig:safe4}
\end{figure}
This suggests that if we split our ensemble of TCLs into two parts - one that receives SP-1 and one that receives SP-2 then we can generate a desired power profile that changes power consumption almost instantly and sustains it at the required level for sufficiently long duration.
In Fig.~\ref{fig:safe_increase}~(left) and Fig.~\ref{fig:switch}~(right)  both power profiles of
SP-1 and SP-2 take about a TCL cycle duration before settling down to a steady state value, and  the power profiles experience peaks at around the same time in opposite directions. Thus if a section of a population follows SP-1, and the rest of TCLs follow SP-2, the aggregate power can be made to follow a square-like pulse profile before settling down after one cycle duration. Of course, the challenge is to design the population size undergoing protocols SP-1 and SP-2, and also the amount of setpoint shift to be used for SP-2. Fig.~\ref{fig:safe4} shows a graphical illustration of how a hybrid SP is supposed to work.

The following parameters are important to quantify and design the protocol:
\begin{itemize}
\item $p$: the fraction of the population, of size $N$, undergoing SP-1 (while the rest of the population, of size $(N-Np)$ undergoes SP-2), and
\item $\delta$: the amount of the temperature shift to be used for SP-2. To generate a negative pulse, $\delta>0$ (i.e. $sign(\delta)=1$) and accordingly in SP-1 all the TCLs are to be switched OFF. Opposite actions are to be taken to generate positive pulse.
\end{itemize}

Fig.~\ref{fig:safe_sw_nonexact} presents two sample simulation results of application of this strategy for generating a positive (Fig.~\ref{fig:safe4_on}) and a negative pulse (Fig.~\ref{fig:safe4_off}).
\begin{figure*}[htp]
  \centering
  \subfigure[Positive pulse: $p=0.3$, $\delta=-0.9~^oC$]{\includegraphics[width=2.65in]{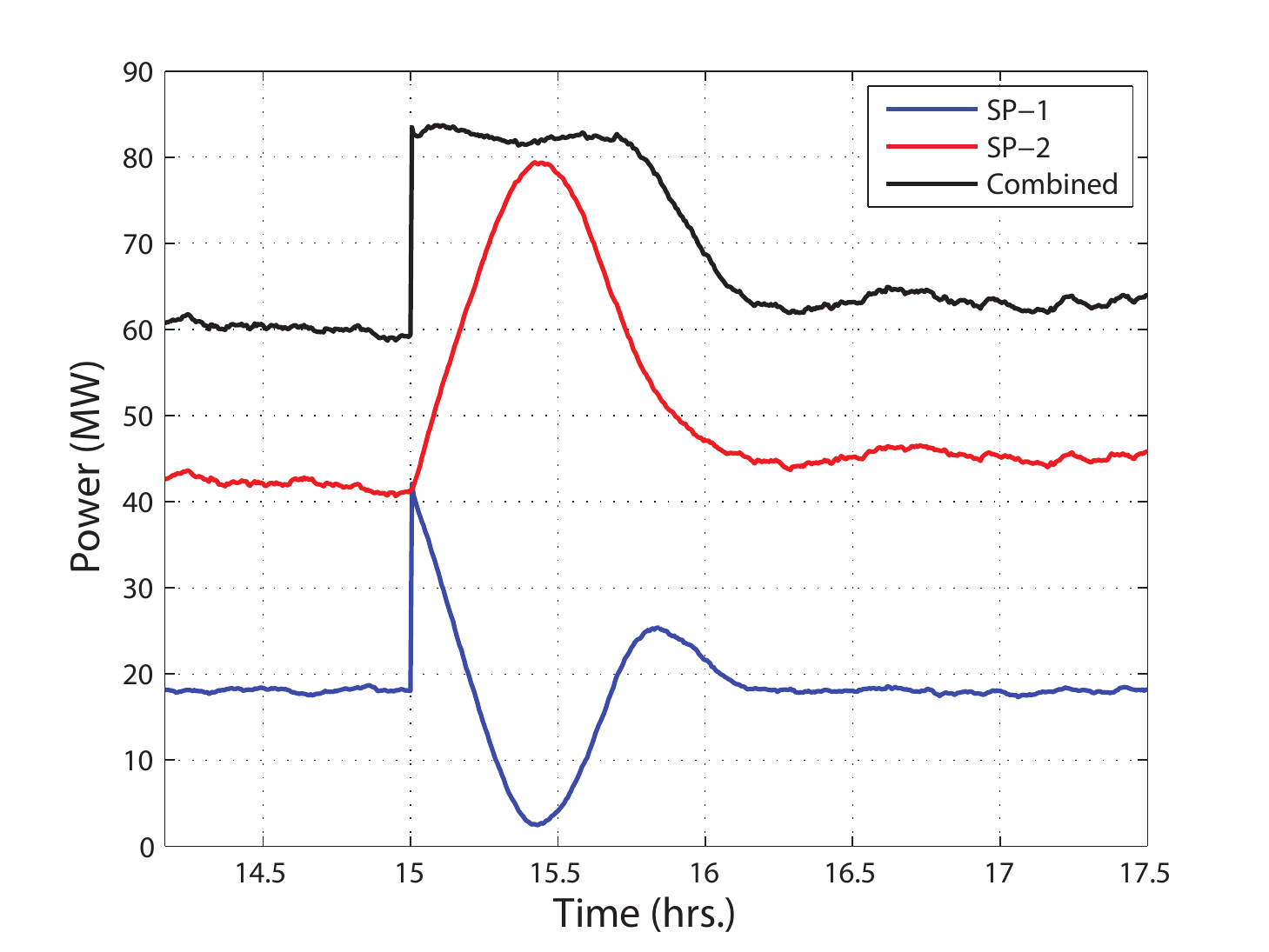}\label{fig:safe4_on}}\hspace{0.01in}
  \subfigure[Negative pulse: $p=0.36$, $\delta=0.9~^oC$]{\includegraphics[width=2.65in]{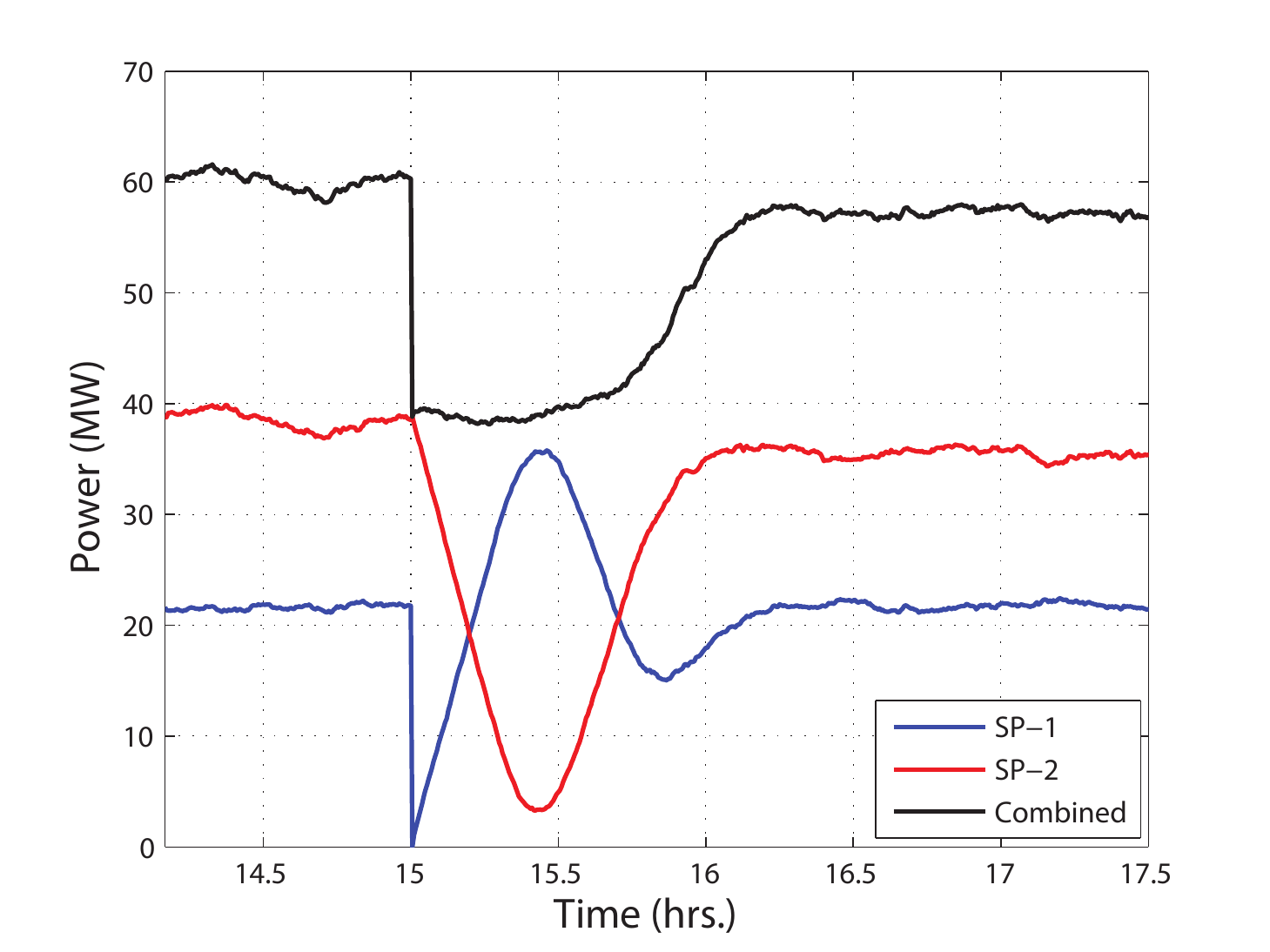}\label{fig:safe4_off}}
  \caption{Power response under hybrid protocol that combines SP-1 and SP-2.}\label{fig:safe_sw_nonexact}
\end{figure*}
Fig.~\ref{fig:safe4_off} shows that by combining two SPs, we can instantly respond to emergency by decreasing power consumed by TCLs and sustain the lower power demand by the population almost at constant value during half an hour. During the following half an hour the power demand by the population is slowly growing and saturating close to the initial value. This is a type of the power profile that is frequently needed from spinning reserves \cite{kirby08,kirby07}.

\section{Conclusion}\label{sec:concl}

We have described and simulated the operation of a class of open-loop safe protocols that  resolve the TCL synchronization problem nearly eliminating oscillations of controlled TCL power.  These safe protocols have several benefits:
\begin{itemize}
\item absence of synchronization of TCLs and related parasitic oscillations
\item fast control response (limited mostly by communication latency) enabling the generation of sharp pulses
\item no requirement for two-way communication, and consequently preservation of privacy
\item very limited information about the TCL ensemble is required to estimate the response. 
\item TCL ensembles can be for both up and down regulation, i.e. TCL power can be both increased and decreased
\end{itemize}

However, safe protocols also appear to have certain limitations whose impact require further investigation:
\begin{itemize}
\item more frequent than normal switchings of TCLs can shorten their lifetime. For example, when SP-3 is used with a small duration $\delta t$, a TCL will undergo at least two switchings while normally it takes a cycle duration (e.g. $\sim$60~min in the studied population) to complete two switchings. This effect likely limits SP-3 to specific applications to responding to unusually large but relatively rare power imbalances such as on the shoulder of wind ramps
\item the response of air conditioners to the control signal is usually restricted by a delay of about 20 seconds, which restricts duration and precision of SP-3 application to this class of TCLs
\item customer tolerance can limit allowable changes in temperature setpoint restricting the application of SP-2
\item the finite size of a TCL population restricts the total energy that can be compensated through pulse generation. After obtaining a useful power pulse, it takes about the natural TCL cycle time for the switched safe protocols to "recover", i.e. to bring switched TCLs back to the customer-defined dead band range.  Therefore, a given TCL can be used only once per this characteristic time
\item an  estimate of at least some aggregated TCL ensemble parameters are still desired for accurate control of generated power response. This information is likely obtainable from monitoring at the utility level, e.g. by observing response from  weak short pulses generated by SP-3. An alternative possibility is to use a {\it limited} number of  houses with  smart meters for two-way communication to provide information that is sufficient to estimate basic aggregated parameters of the whole population   

\item the safe protocols require a small amount of trivial to realize embedded intelligence in each TCL controller. 
\end{itemize}

If these challenges can be overcome, the potential impact of TCL demand response, and specifically the safe protocols described here, would be large. Current estimates show that there are over 130 millions  houses in the USA alone with over 100 millions of them equipped with air conditioners \cite{acnum}.  Considering a conservative estimate of 100MW reliable power control from air conditioners  and a similar amount from water heaters in a one million houses,   this technology  implemented on the national scale could provide up to 20 GW power of ancillary services.

\section*{Acknowledgment}
The authors would like to thank Professor Ian Hiskens of University of Michigan, Ann Arbor for his valuable guidance throughout this work. The work at LANL was carried out under the auspices of the National Nuclear
Security Administration of the U.S. Department of Energy at Los
Alamos National Laboratory under Contract No. DE-AC52-06NA25396. 


\appendix

\section{}
\begin{theorem}
\emph{(No-Go Theorem)}
One cannot generate  extra energy by switching conditioners between ON and OFF branches along the SP-1 in a large heterogeneous ensemble of TCLs.
\end{theorem}

\begin{proof}
Consider application of the SP-1 to a large population of TCLs, that are initially in the uncorrelated steady state. We assume that due to large size of the population, it can be partitioned by a number of sub-populations, which still have large sizes, with practically identical parameters of TCLs and ambient conditions. It is then sufficient to prove the No-Go Theorem for an arbitrary homogeneous sub-population, because if any subpopulation does not produce an excess of energy then neither does their ensemble.

The evolution of temperature of a TCL, in (\ref{eq:micro}), is linear both in temperature and consumed power. Hence, after averaging (\ref{eq:micro}) over the population with random initial phases, we find that 
the averaged consumed from the Power Grid power $\langle P \rangle$, and the averaged temperature $ \langle \theta \rangle$ of the ensemble also satisfy a linear relation:
\begin{equation}
\alpha \frac{d}{dt} \langle \theta \rangle=-\langle P(t)\rangle+\beta(\langle \theta \rangle - \gamma),
\label{ev1}
\end{equation}
where $\alpha$, $\beta$, and $\gamma$ are constants that characterize the ensemble. They can be expressed via the number of TLCs $N$, the efficiency of the TCLs, and parameters $R$, $C$ and $\theta_{amb}$. At the steady state, we would have $ \frac{d}{dt} \langle \theta \rangle=0$, so that the average power consumed by the ensemble in the steady state is $\langle P_{\rm st} \rangle=\beta(\langle \theta_{st} \rangle - \gamma)$,  where $\langle \theta_{\rm st} \rangle$ is the average temperature in the unperturbed ensemble of TCLs. The total energy consumed by the ensemble in  the steady state during some time $\tau$ then would be 
\begin{equation}
W_{st}(\tau)=\beta(\langle \theta \rangle - \gamma) \tau.
\label{Wst}
\end{equation}

Consider now the total energy, $W_{\rm tot}$, generated by the TCL ensemble during the time $t_p$ between the initial and the final stages of the SP-1:
\begin{equation}
W_{\rm tot} = \int_0^{t_p} dt \langle P(t) \rangle =
\beta\int_0^{t_p} dt (\langle \theta \rangle - \gamma),
\label{w2}
\end{equation}
where we used (\ref{ev1}) and the fact from Section III that, for SP-1, at time $t_p$,  the final distribution of TCLs over temperature and ON/OFF states coincides with the initial distribution. This means that the average temperature of the ensemble coincides with the initial one, so that integrating the lhs of (\ref{ev1}) we obtain
$\int_0^{t_p} \frac{d}{dt} \langle \theta \rangle dt =0$.

Here we note that the time $t_p$ coincides with the period of the cycle of a single TCL. Indeed, if we trace the evolution of any TCL of the ensemble during the SP-1 stages, we find that it goes purely autonomously during the period $t_p$ through all stages of the considered temperature operation band, just like TCLs in nonperturbed ensemble except switching the branches at the beginning and the end of the protocol. 
We also note that at the steady state, the averaged over time power consumed by one TCL during the  cycle of its operation coincides with the power averaged over the large population of TCLs with the random phase distribution. This means that, for a single TCL in a homogeneous population, $\int_0^{t_p} \theta dt = \langle \theta_{\rm st} \rangle t_p$.
 Hence, (\ref{w2}) can be rewritten as
 \begin{equation}
W_{\rm tot} = \beta( \langle \theta_{\rm st} \rangle - \gamma) t_p ,
\label{w22}
\end{equation}
But (\ref{w22}) is exactly the energy $W_{\rm st}$ in (\ref{Wst}) consumed by the TCL ensemble during time $t_p$ without application of the external control. Hence the excess energy is $W_{\rm tot}-W_{\rm st}=0$.
\end{proof}



\bibliographystyle{model1-num-names}
 


\end{document}